\documentclass{article}
\usepackage[T1]{fontenc}
\usepackage{amsfonts}
\usepackage{amssymb}
\usepackage{amsmath}
\usepackage{amsthm}
\usepackage[a4paper]{geometry}
\usepackage{subfigure}
\usepackage{comment}

\usepackage{tikz}
\usetikzlibrary{snakes}
\usetikzlibrary{arrows}
\usetikzlibrary{shapes}

\newtheorem{theorem}{Theorem}[section]
\newtheorem{lemma}[theorem]{Lemma}
\newtheorem{proposition}[theorem]{Proposition}

\theoremstyle{definition}

\newtheorem{remark}[theorem]{Remark}

\newcommand{\ud}{\mathrm{d}}
\newcommand{\cee}{\mathbb{C}}
\newcommand{\er}{\mathbb{R}}
\newcommand{\lam}{\lambda}
\newcommand{\enn}{\mathbb{N}}

\newcommand{\bol}{\hfill\square\\}

\newcommand{\til}{\tilde}

\numberwithin{equation}{section}

\title{Multiple Meixner-Pollaczek polynomials \\ and the six-vertex model}

\author{Martin Bender\footnotemark[1],\quad Steven Delvaux\footnotemark[2],\quad Arno B.~J.~Kuijlaars\footnotemark[2]}
\date{\today}

\begin{document}

\maketitle
\renewcommand{\thefootnote}{\fnsymbol{footnote}}
\footnotetext[1]{ MSRI, 17 Gauss Way, Berkeley, CA 94720-5070. email:
mbender@msri.org} \footnotetext[2]{Department of Mathematics, Katholieke
Universiteit Leuven, Celestijnenlaan 200B, B-3001 Leuven, Belgium. email:
\{steven.delvaux, arno.kuijlaars\}\symbol{'100}wis.kuleuven.be.\\  Steven
Delvaux is a Postdoctoral Fellow of the Fund for Scientific Research - Flanders
(Belgium). Arno Kuijlaars is supported by K.U. Leuven research grant OT/08/33,
FWO-Flanders projects G.0427.09 and G.0641.11, by the Belgian Interuniversity
Attraction Pole P06/02, and by grant MTM2008-06689-C02-01 of the Spanish
Ministry of Science and Innovation. }

\begin{abstract}
We study multiple orthogonal polynomials of Meixner-Pollaczek type with respect
to a symmetric system of two orthogonality measures. Our main result is that
the limiting distribution of the zeros of these polynomials is one component of
the solution to a constrained vector equilibrium problem. We also provide a
Rodrigues formula and closed expressions for the recurrence coefficients. The
proof of the main result follows from a connection with the eigenvalues of
block Toeplitz matrices, for which we provide some general results of independent interest. 

The motivation for this paper is the study of a model in statistical mechanics,
the so-called six-vertex model with domain wall boundary conditions, in a
particular regime known as the free fermion line. We show how the multiple
Meixner-Pollaczek polynomials arise in an inhomogeneous version of this model.

\end{abstract}

\section{Introduction}
\label{section:introduction}

In this paper we study a system of polynomials, orthogonal with
respect to two different weight functions of Meixner-Pollaczek type. Our work
is motivated by the analysis of the six-vertex model in statistical mechanics
introduced in \cite{Kor} and studied in many papers since then, see e.g.\ 
\cite{Bax,ICK,Iz,Kor,Kup,KZ,Zinn}. Colomo and Pronko \cite{CP1,CP2} 
studied the role of orthogonal polynomials in the six-vertex model, and
in particular the Meixner-Pollaczek polynomials. 
An object of major physical interest is the
\emph{partition function} of this model. A rather complete analysis of the
large $n$ asymptotics in the homogeneous case is based on the analysis of a
Riemann-Hilbert problem for orthogonal polynomials \cite{BF,BL1,BL2,BL3}. In the
inhomogeneous case, the corresponding reasoning leads to questions of
asymptotics of \emph{multiple orthogonal} polynomials, as we will discuss in
Section~\ref{section:sixvertex}.

Here we will focus on a particular regime of the six-vertex model, known as the
``free fermion line''. In the homogeneous case, this case is trivial, the
partition function being identically $1$. Also in the inhomogeneous case, a
closed expression for the partition function can be calculated explicitly, so
the present results do not give any new insights into the original model, but
should rather be considered as providing exact and asymptotic information in its own right on the
associated polynomials, the \emph{multiple
Meixner-Pollaczek polynomials}.

An important tool in the study of (usual) orthogonal polynomials $P_n$ on
the real line, appropriately scaled in order to have $n$-dependent weights, is
that their limiting zero distribution satisfies an equilibrium problem. This
equilibrium problem is an important ingredient for the steepest descent analysis of the
Riemann-Hilbert problem for orthogonal polynomials, thereby allowing to obtain
strong and uniform asymptotics of the polynomials, see e.g.\
\cite{DKMVZ1,DKMVZ2}.

For the case of multiple orthogonal polynomials, however, no general result is
known about the existence of an equilibrium problem. The aim of this paper is
to obtain such an equilibrium problem for multiple Meixner-Pollaczek
polynomials. The equilibrium problem will be posed in terms of a couple of
measures $(\nu_1,\nu_2)$ and it involves both an external source acting on
$\nu_1$ and a constraint acting on $\nu_2$. This structure is very similar to
the equilibrium problem for the GUE with external source model, \cite{BK}.

Our method for obtaining the equilibrium problem is similar to that applied to
other systems of multiple orthogonal polynomials \cite{Geudens,Roman}, but a
distinguishing feature is the characterization in terms of the eigenvalue
distributions of block Toeplitz matrices, rather than usual (scalar) Toeplitz
matrices. We will provide some general results on this topic that are of
independent interest. Along the way we also obtain Rodrigues formulae and
closed expressions for the recurrence coefficients for multiple
Meixner-Pollaczek polynomials.

Inspired by the scalar case \cite{DKMVZ1,DKMVZ2}, one might hope that the
equilibrium problem presented in this paper can be used in the steepest descent
analysis of the Riemann-Hilbert problem for multiple Meixner-Pollaczek
polynomials, thereby obtaining strong and uniform asymptotics for these
polynomials. This approach, although interesting, will not be carried out in
this paper. We also hope that our equilibrium problem might serve as an
inspiration to obtain similar results for the general inhomogeneous six-vertex
model, not necessarily on the free fermion line, which in turn could serve as
the first step in finding the large $n$ asymptotic analysis for this model.

\section{Statement of results}
\label{section:statementresults}

\subsection{Multiple Meixner-Pollaczek polynomials}
\label{subsection:MMPP}

Let $w_1$ and  $w_2$ denote two distinct weight functions of Meixner-Pollaczek type; $w_j:\er\to\er^+$ with
\begin{equation}\label{def:w12}
w_j(x)=\frac{1}{2\pi}e^{2t_jx}\left|\Gamma\left(\lambda+ix\right)\right|^2,\quad
j=1,2,
\end{equation}
where $t_1,t_2\in (-\pi/2,\pi/2)$, $t_1\ne t_2$ and $\lambda>0$ are fixed
parameters and $\Gamma$ denotes Euler's gamma function
\cite{AS}. Note that in \eqref{def:w12} the gamma function is evaluated in a
complex argument and that
$$\left|\Gamma\left(\lambda+ix\right)\right|^2 =
\Gamma\left(\lambda+ix\right)\Gamma\left(\lambda-ix\right),\quad \textrm{for
any }\lambda,x\in\er.$$ Furthermore, for $\lambda$ fixed, $\left|\Gamma\left(\lambda+ix\right)\right|^2\sim e^{-\pi |x|}$ as $|x|\to \infty$, so the restrictions on $t_j$ guarantee that $w_j$ is exponentially decaying for $x\to \pm \infty$.

\begin{lemma}\label{lemma:interlacing} (Existence, uniqueness, real and interlacing zeros)
For any nonnegative integers $k_1$ and $k_2$, there is
a unique monic polynomial $P_{k_1,k_2} $  of degree $k=k_1+k_2$ satisfying
the orthogonality conditions
\begin{equation*}
\int_{-\infty}^{\infty} P_{k_1,k_2}(x)x^m w_j(x)\ud x= 0, \quad\textrm{ for
}m=0,\ldots, k_j-1,\quad j=1,2.
\end{equation*}

The zeros of these polynomials are  
 real and interlacing, in the sense that each
$P_{k_1,k_2}$ has $k$ distinct real
zeros $x_1^{k_1,k_2}<x_2^{k_1,k_2}<\cdots < x_k^{k_1,k_2}$ such that
$x_j^{k_1,k_2}<x_j^{k_1-1,k_2}<x_{j+1}^{k_1,k_2}$  and $x_j^{k_1,k_2}<x_j^{k_1,k_2-1}<x_{j+1}^{k_1,k_2}$
whenever $1\leq j\leq k-1$.
\end{lemma}

Lemma~\ref{lemma:interlacing} will be proved in
Section~\ref{subsection:proofinterlacing}. 

In analogy with the case of Meixner-Pollaczek orthogonal polynomials
\cite{Koekoek} we refer to the $P_{k_1,k_2}$ as \emph{multiple
Meixner-Pollaczek polynomials}; see also \cite[Sec.~4.3.3]{BCV}. These
polynomials are related to the six-vertex model; see
Section~\ref{section:sixvertex}. For information on other systems of multiple
orthogonal polynomials in the literature, see e.g.\ \cite{Apt,VA}.

In this paper we will derive a Rodrigues type formula for the polynomials
$P_{k_1,k_2}$, enabling us to compute explicitly the following four term
recurrence relations, to be proved  in
Section~\ref{section:recurrence}.

\begin{theorem}\label{theorem:recurrenceMMPP} (Recurrence relations)
Let $t_1, t_2 \in (-\pi/2, \pi/2)$ with $t_1 \neq t_2$. Then, for nonnegative
integers $k_1$ and $k_2$, the multiple Meixner-Polleczek polynomials satisfy
the recurrence relations
\begin{equation}\label{recurrenceMP}
P_{k_1+1,k_2}(x)
=\left(x-a_{k_1,k_2}^{t_1,t_2}\right)P_{k_1,k_2}(x)-b_{k_1,k_2}^{t_1,t_2}P_{k_1,k_2-1}(x)
-c_{k_1,k_2}^{t_1,t_2}P_{k_1-1,k_2-1}(x),
\end{equation}
and
\begin{equation}\label{recurrenceMP2}
P_{k_1,k_2+1}(x)
=\left(x-a_{k_2,k_1}^{t_2,t_1}\right)P_{k_1,k_2}(x)
-b_{k_2,k_1}^{t_2,t_1}P_{k_1-1,k_2}(x)-c_{k_2,k_1}^{t_2,t_1}P_{k_1-1,k_2-1}(x),
\end{equation}
where
\begin{equation*}
a_{k_1,k_2}^{t_1,t_2}=\frac{(k+k_1+2\lambda)}{2}\tan t_1+\frac{k_2}{2}\tan t_2,
\end{equation*}
\begin{equation*}
b_{k_1,k_2}^{t_1,t_2}=\frac{(k+2\lambda-1)}{4}\left(\frac{k_1}{\cos^2  \! t_1}
+\frac{k_2}{\cos^2  \! t_2}\right),
\end{equation*}
\begin{equation*}
c_{k_1,k_2}^{t_1,t_2}=\frac{k_1(k+2\lambda-1)(k+2\lambda-2)(\tan t_1-\tan
t_2)}{8\cos^2  \! t_1},
\end{equation*}
$k=k_1+k_2$, and where we set $P_{-1,k_2}\equiv 0$, $P_{k_1,-1}\equiv 0$
for any $k_1,k_2$.
\end{theorem}

The above theorem is the basis for the main purpose of the paper, namely to
study the asymptotic zero distribution of the appropriately rescaled multiple
Meixner-Pollaczek polynomials. Fix $ t_1$ and $t_2$. For simplicity we choose a
particular sequence of indices $(k_1(k),k_2(k))_{k=1}^ {\infty}$ along which
we analyze the large $k$ asymptotics of $P_{k_1(k),k_2(k)}$,
 and form the single sequence
\begin{equation}
Q_{k}(x)=\left\{\begin{array}{ll}P_{k/2,k/2}(x) & \textrm{ if $k$ is even,}\nonumber \\
P_{(k+1)/2,(k-1)/2}(x) & \textrm{ if $k$ is odd,}
\end{array}\right.
\end{equation}
of polynomials. We will show that the zero-distribution of the rescaled
polynomials $Q_k(kx)$ has a weak limit as $k$ goes to infinity, and in order to
study this we let $n\in \mathbb{N}$ and  introduce the doubly indexed sequence
of monic polynomials $\{Q_{k,n}\}_{k\geq 0}$ defined by
$$Q_{k,n}(x)=\frac{1}{n^{k}}Q_{k}(nx).$$
From
(\ref{recurrenceMP}) and \eqref{recurrenceMP2}
we obtain an explicit recurrence
relation for $Q_{k,n}$,
\begin{equation}\label{recurrencebis}
xQ_{k,n}(x)=Q_{k+1,n}(x)+a_{k,n}Q_{k,n}(x)+b_{k,n}Q_{k-1,n}(x)+c_{k,n}Q_{k-2,n}(x),
\end{equation}
with initial conditions $Q_{-3,n}\equiv Q_{-2,n}\equiv Q_{-1,n}\equiv 0$, where
\begin{equation}\label{MMOP:reccoeffeven}
\left\{
\begin{array}{lll}
\displaystyle a_{k,n}=\frac{1}{n}a_{k/2,k/2}^{t_1,t_2}=\frac{3k+4\lambda}{4n}\tan t_1+\frac{k}{4n}\tan t_2 \\
\displaystyle b_{k,n}=\frac{1}{n^2}b_{k/2,k/2}^{t_1,t_2}=\frac{k(k+2\lambda-1)}{8n^2}\left(\frac{1}{\cos^2  \! t_1}+\frac{1}{\cos^2  \! t_2}\right)\\
\displaystyle c_{k,n}=\frac{1}{n^3}c_{k/2,k/2}^{t_1,t_2}=\frac{k(k+2\lambda-1)(k+2\lambda-2)}{16n^3}\frac{(\tan t_1-\tan t_2)}{\cos^2  \! t_1}
\end{array} 
\right.
\end{equation}
for  $k$ even and
\begin{equation}\label{MMOP:reccoeffodd}
\left\{
\begin{array}{lll}
a_{k,n}=\displaystyle \frac{1}{n}a_{(k-1)/2,(k+1)/2}^{t_2,t_1} =\frac{k+1}{4n}\tan t_1 +\frac{3k+4\lambda-1}{4n}\tan t_2\\
b_{k,n}=\displaystyle \frac{1}{n^2}b_{(k-1)/2,(k+1)/2}^{t_2,t_1}=\frac{k+2\lambda-1}{8n^2}\left(\frac{k+1}{\cos^2  \! t_1}+\frac{k-1}{\cos^2  \! t_2}\right)\\
c_{k,n}=\displaystyle \frac{1}{n^3}c_{(k-1)/2,(k+1)/2}^{t_2,t_1}=\frac{(k-1)(k+2\lambda-1)(k+2\lambda-2)}{16n^3}\frac{(\tan t_2-\tan t_1)}{\cos^2  \! t_2}
\end{array} 
\right.
\end{equation}
for $k$ odd.

Using the recurrence relation \eqref{recurrencebis}, 
standard considerations show that the zeros of $Q_{n,n}$ are the eigenvalues of
the $4$-diagonal matrix
\begin{equation}\label{jacobi}
\left(
\begin{array}{ccccc}
 a_{0,n}& 1 &           &     & 0     \\
b_{1,n} &   a_{1,n}  &  1 &         &  \\
  c_{2,n}    & b_{2,n}    & a_{2,n}  &  \ddots   & \\
 &  \ddots   &  \ddots    &\ddots   & 1\\
0       &    &  c_{n-1,n}    & b_{n-1,n}   &a_{n-1,n}
\end{array}
\right)_{n\times n}.
\end{equation}
So the problem of the asymptotic zero-distribution of $Q_{n,n}$ reduces to finding the eigenvalue asymptotics of (\ref{jacobi}).
If $k$ and $n$ 
both tend to infinity in such a
way that $k/n\to s$ for some constant $s$, then the coefficients $a_{k,n}$
have two subsequential limits $a_s$ and $\til a_s$ for even and odd $k$ respectively. 
Similarly, $b_{k,n}$ and
$c_{k,n}$ have subsequential limits $b_{s}, c_{s}$ and $\til b_{s}, \til
c_{s}$ along subsequences consisting of even and odd $k$ respectively.
Using the identity $1/\cos^2  \! t=1+\tan^2  \! t$, these limits become
\begin{equation}\label{limits:abc:nonsymm}
\left\{
\begin{array}{l}
a_s= \left(3\tan t_1+\tan t_2\right)s/4 \\
b_s= \left(2+\tan^2  \! t_1+\tan^2  \! t_2\right)s^2/8\\
c_s= (\tan t_1-\tan t_2)(1+\tan^2  \! t_1)s^3/{16},
\end{array}
\right.
\end{equation}
\begin{equation}\label{limits:abc:nonsymm2}
\left\{
\begin{array}{l}
\til a_s= \left(3\tan t_2+\tan t_1\right)s/4 \\
\til b_s= b_s \\
\til c_s= (\tan t_2-\tan t_1)(1+\tan^2  \! t_2)s^3/{16}.
\end{array}
\right.
\end{equation}
If we consider (\ref{jacobi}) in $2\times2$ blocks, it is tri-diagonal with blocks
\begin{equation*}
A^{(-1)}_{k,n}=\left( \begin{array}{cc}
 0& 0 \\
1&0
  \end{array}\right),
\end{equation*}
\begin{equation*}
A^{(0)}_{k,n}= \begin{pmatrix}
 a_{2k,n}& 1\\
 b_{2k+1,n}& a_{2k+1,n}
  \end{pmatrix},
\end{equation*}
and
\begin{equation*}
A^{(1)}_{k,n}= \begin{pmatrix}
 c_{2k,n}& b_{2k,n} \\
0&c_{2k+1,n}
  \end{pmatrix}.
\end{equation*}
In view of the limits \eqref{MMOP:reccoeffeven} and \eqref{MMOP:reccoeffodd},
the blocks are slowly varying along the diagonals if $n$ is large.
In other words, for large $n$, the matrix (\ref{jacobi}) has a \emph{locally block Toeplitz} structure with square blocks of size $r=2$.
This allows its limiting eigenvalue distribution
 to be obtained from a general machinery to which we turn
now.
\subsection{Polynomials generated by a general recurrence relation}
\label{subsection:eigloctoep}

In this subsection we will work in the following general setting. Let $n$ be a
fixed parameter and let $(Q_{k,n}(x))_{k=0}^{\infty}$ be a sequence of monic
polynomials, where $Q_{k,n}$ has degree $k$ and depends parametrically on $n$.
Assume that the $Q_{k,n}$ are generated by the recurrence relation
\begin{equation}\label{reccoef:Jacobi}
x\begin{pmatrix} Q_{0,n}(x) \\ Q_{1,n}(x) \\ \vdots
\end{pmatrix}
= J_n
\begin{pmatrix} Q_{0,n}(x) \\ Q_{1,n}(x) \\ \vdots
\end{pmatrix},
\end{equation}
where $J_n$ is a semi-infinite matrix with unit lower Hessenberg structure,
i.e., the strictly upper triangular part of $J_n$ is equal to zero, except for
the first superdiagonal, on which all entries are $1$. We also assume that the
lower triangular part of $J_n$ has a finite bandwidth, which is independent of
$n$.

The entries of $J_n$ are assumed to have asymptotically periodic behavior with period $r$
($r\geq 1$). More precisely, suppose $J_n$ is partitioned into blocks of
size $r\times r$, with one superdiagonal and a fixed finite number $\beta$ of subdiagonal non-zero blocks,

\begin{equation}\label{J:part} J_n = \left( A_{k,n}^{(k-l)} \right)_{k,l=0}^{\infty} =
\begin{pmatrix}
A_{0,n}^{(0)}         &    A_{0,n}^{(-1)}       &              0          &     \cdots  \\
A_{1,n}^{(1)}         &    A_{1,n}^{(0)}        & A_{1,n}^{(-1)}          &     \ddots  \\
   \vdots           &    \ddots               & \ddots                  &     \ddots  \\
           \vdots           &        \vdots           &  \vdots                 &     \vdots          \\
\vdots                &  \ddots              &              &     \\
A_{\beta,n}^{(\beta)} & A_{\beta,n}^{(\beta-1)} & A_{\beta,n}^{(\beta-2)} &           \\
0                     & A_{\beta+1,n}^{(\beta)} &A_{\beta+1,n}^{(\beta-1)}& \ddots     \\
\vdots                & \ddots                  & \ddots                  & \ddots
\end{pmatrix},
\end{equation}
where
\begin{equation}\label{A0minus1:structure} A_{k,n}^{(-1)} =
\begin{pmatrix}
0 & 0& \cdots & 0 \\
\vdots & \vdots & & \vdots \\
0 & 0 & \cdots & 0 \\
1 & 0 & \cdots & 0 \\
\end{pmatrix}_{r\times r},\qquad A_{k,n}^{(0)} =
\begin{pmatrix}
* & 1 & & 0 \\
\vdots &  & \ddots & \\
* &  & & 1 \\
* & * & \cdots & * \\
\end{pmatrix}_{r\times r},
\end{equation}
where the $*$'s denote arbitrary constants.

We will assume that the block entries $A_{k,n}^{(j)}$ in \eqref{J:part} are
slowly varying with $n$, in the sense that the limits
\begin{equation}\label{reccoef:Toeplitz} \lim_{rk/n\to s}
A_{k,n}^{(j)}=A^{(j)}_{s}
\end{equation}
exist for any $j=-1,0,\ldots,\beta$ and $s\geq 0$. Here the notation
$\lim_{rk/n\to s}$ means that we let both $k$ and $n$ tend to infinity, in such
a way that the ratio $rk/n$ tends to a limit $s\geq 0$. The relation
\eqref{reccoef:Toeplitz} is on the level of $r\times r$ matrices, with the
limit taken entry-wise. If \eqref{reccoef:Toeplitz} holds then the matrix in
\eqref{J:part} is said to have \emph{locally block Toeplitz} structure, in the
spirit of \cite{Tilli}.

For fixed $s$, we collect the limiting matrices in
\eqref{reccoef:Toeplitz} into the following matrix-valued Laurent polynomial $A_s(z)$:
\begin{equation}\label{symbol0} A_s(z) := A_{s}^{(-1)}z^{-1}+
A_s^{(0)}+\cdots+ A_{s}^{(\beta)} z^{\beta}.\end{equation}
We will sometimes refer to $A_s(z)$ as the \emph{symbol}. From
\eqref{A0minus1:structure}--\eqref{symbol0} it follows that
\begin{equation}\label{Fzlambda0}
A_s(z)=\begin{pmatrix}
* & 1 & & 0 \\
\vdots &  & \ddots & \\
* &  & & 1 \\
(z^{-1}+*) &
* & \cdots & * \\
\end{pmatrix}+\mathcal{O}(z),
\end{equation}
where  
$\mathcal{O}(z)$ denotes all the terms in \eqref{symbol0} that tend to zero as $z\to 0$.

Define for each $s \geq 0$, the algebraic equation
\begin{equation}\label{fzlambda0}
f_s(z,x):=\det(A_s(z)-x I_r)=0,
\end{equation}
where $I_r$ denotes the identity matrix of size $r$. Note that $f_s$ depends on
two complex variables $z$ and $x$. By expanding the determinant
\eqref{fzlambda0} and using \eqref{Fzlambda0}, we can write $f_s$ as a (scalar)
Laurent polynomial in~$z$:
\begin{equation}\label{f:expansion0} f_s(z,x) = f_{-1,s}(x) z^{-1}+f_{0,s}(x)+\cdots +
f_{p,s}(x)z^p,
\end{equation}
where $f_{j,s}(x)$, $j=-1,0,\ldots,p$, are polynomials in $x$, with
$f_{-1,s}(x)\equiv (-1)^{r-1}$. We define $p$ in \eqref{f:expansion0} as the
largest positive integer for which $f_{p,s}\not\equiv 0$.

Let us solve \eqref{fzlambda0} for $z$; this yields $p+1$ roots
$$ z_j = z_{j}(x,s),\qquad j=1,\ldots,p+1.$$
We assume that for each $x\in\cee$ these roots are ordered such that
\begin{equation}\label{ordering:roots0} 0 < |z_1(x,s)|\leq |z_2(x,s)|\leq\cdots\leq
|z_{p+1}(x,s)|.
\end{equation}
If $x\in\cee$ is such that two or more consecutive roots in
\eqref{ordering:roots0} have the same absolute value, then we arbitrarily
label them so that \eqref{ordering:roots0} is satisfied.

Define the set
\begin{equation}\label{def:Gamma0} \Gamma_1(s) = \{x\in\cee\mid |z_{1}(x,s)| = |z_{2}(x,s)|
\}.
\end{equation}
In the cases we are interested in, we will have that
\begin{equation}\label{assumption:realline}\Gamma_1(s)\subset\er.\end{equation}
Supposing this to hold, we define a measure $\mu_1^s$ on $\Gamma_1(s)\subset\er$ with density
\begin{equation} \label{def:measures0} \ud \mu_1^s(x) = \frac 1r\frac{1}{2\pi i}\left(
\frac{z_{1+}'(x,s)}{z_{1+}(x,s)}-\frac{z_{1-}'(x,s)}{z_{1-}(x,s)} \right)\ud x.
\end{equation}
Here the prime denotes the derivative with respect to $x$,
and $z_{1\pm}(x,s)$ are the boundary values of $z_1(x,s)$ obtained from the
$+$-side (upper side) and $-$-side (lower side) respectively of
$\Gamma_1(s)\subset\er$. These boundary values exist for all but a finite number
of points.

As discussed in \cite{Delvaux}, the measure $\mu_1^s$ can be
interpreted as the weak limit as $n\to\infty$ of the normalized eigenvalue
counting measures for the block Toeplitz matrices $T_n(A_s)$ associated to the
symbol \eqref{symbol0}.
\begin{lemma}\label{lemma:mu0prob} With the above notation, we have
\begin{itemize}
\item[\rm{(a)}] $z_1(x,s)= x^{-r}+\mathcal{O}\left(x^{-r-1}\right)$ as $x\to\infty$.
\item[\rm{(b)}]$\mu_1^s$ in \eqref{def:measures0} is a Borel probability measure on
$\Gamma_1(s)$.
\end{itemize}
\end{lemma}

\begin{proof}
See \cite{Delvaux}.
\end{proof}

With this notation in place, let us return to the sequence of polynomials
$(Q_{k,n})_{k=0}^{\infty}$ in \eqref{reccoef:Jacobi}. This sequence is said to
have \emph{real and interlacing} zeros if each $Q_{k,n}$ has $k$ distinct real
zeros $x_1^{k,n}<x_2^{k,n}<\cdots<x_k^{k,n}$ such that
$x_j^{k,n}<x_j^{k-1,n}<x_{j+1}^{k,n}$ whenever $1\leq j\leq k-1$.

The next result states that, under certain conditions, the normalized
zero-counting measures of the polynomials $Q_{n,n}$ have a (weak) limit for
$n\to\infty$, which is  precisely the average of the measures
\eqref{def:measures0}. Here the average is with respect to the parameter $s$.

\begin{theorem}\label{theorem:locallyToepeig} (Limiting zero distribution of $Q_{n,n}$)
Let the sequence of polynomials $(Q_{k,n})_{k=0}^{\infty}$ be such that
\eqref{reccoef:Jacobi}--\eqref{reccoef:Toeplitz} hold. Assume that
$(Q_{k,n})_{k=0}^{\infty}$ has real and interlacing zeros for each $n$, as described above.
Also assume that \eqref{assumption:realline} holds for every $s\geq 0$. Then as
$n\to \infty$, the normalized zero-counting measure $\rho_n$ of the polynomial
$Q_{n,n}$,
\begin{equation} \label{def:rho}
\rho_n:=\frac{1}{n}\sum_{j=1}^n\delta_{x_j^{n,n}},
\end{equation}
where $\delta_z$ denotes the Dirac point mass at $z$, has the limit
\begin{equation}\label{def:nu_1}
\lim_{n\to \infty} \rho_n=\nu_1:=\int_{0}^1\mu_1^s\, \ud s
\end{equation}
 in the sense of weak convergence of measures,
where $\mu_1^s$ is defined in \eqref{def:measures0}.
\end{theorem}

Theorem~\ref{theorem:locallyToepeig} will be proven in
Section~\ref{section:zeroasymptotics}. This theorem generalizes a result
for the scalar case $r=1$ by Kuijlaars-Rom\'an \cite[Theorem 1.2]{Roman}; see
also \cite{Coussement,KVA}.

%
%

\subsection{Zeros of multiple Meixner-Pollaczek polynomials}
\label{subsection:zerosMMOP}

We apply Theorem~\ref{theorem:locallyToepeig} to the polynomials
$Q_{k,n}$ in Section~\ref{subsection:MMPP}. In what follows, we will assume the
condition of symmetric weights, $$t:=t_1=-t_2;$$ for this case we can
characterize the limiting zero distribution in terms of a constrained vector
equilibrium problem,  Theorem \ref{theorem:equilibriumproblem}, which is the
main result of the paper. The symmetry condition in  Proposition
\ref{theorem:zeroseigenvalues} is needed  to prove that $\Gamma_1(s) \subset
\mathbb{R}$; in the general case this may fail.

The recurrence coefficients in
\eqref{limits:abc:nonsymm} and  \eqref{limits:abc:nonsymm2} then become
\begin{equation}\label{limits:abc:symm}
\left \{
\begin{aligned}
a_s &=& -\til a_s &=(\tan t) s/2, \\
b_s &=&  \til b_s  &= (1+\tan^2  \! t)s^2/4,\\
c_s &=&  - \til c_s & = \tan t(1+\tan^2  \! t)s^3/8=a_s b_s.
\end{aligned}
\right.
\end{equation}

We partition the matrix $J_n$ from \eqref{jacobi} into blocks as in \eqref{J:part} with $r=2$
and $\beta=1$. Then the limiting values in \eqref{reccoef:Toeplitz} exist and
are given by
\begin{equation*}
A^{(-1)}_s=
\begin{pmatrix}
0& 0 \\
1&0
\end{pmatrix},
\end{equation*}
\begin{equation*}
A^{(0)}_s=\begin{pmatrix}
 a_s& 1\\
 b_s& -a_s
  \end{pmatrix},
\end{equation*}
and
\begin{equation*}
A^{(1)}_s=\begin{pmatrix}
 c_s& b_s \\
0&-c_s
  \end{pmatrix},
\end{equation*}
with $a_s$, $b_s$ and $c_s$ as in \eqref{limits:abc:symm}. The symbol $A_s(z)$
in \eqref{symbol0} now becomes
\begin{equation*}\label{symboldef} A_s(z)=
        A_s^{(-1)} z^{-1} + A_s^{(0)} + A_s^{(1)}z = \begin{pmatrix}
 a_s(1+b_s z)& 1+b_sz \\
z^{-1}\left(1+b_s z\right)&-a_s(1+b_s z)
  \end{pmatrix},
\end{equation*}
 so \eqref{fzlambda0} reduces to
\begin{equation}\label{algebraicequation:P}
x^2-\frac{(1+b_s z)^2(1+a_s^2 z)}{z}=0.
\end{equation}
This equation 
has three roots, $z_j(x,s)$, $j=1,2,3$
which we label in order of increasing modulus:
\begin{equation}\label{roots}
0< |z_1(x,s)|\leq |z_2(x,s)|\leq |z_3(x,s)|.
\end{equation}

\begin{proposition}\label{theorem:zeroseigenvalues} (Limiting zero distribution)
 Suppose $t_1=-t_2=t\in (0,\pi/2)$. As $n\to \infty$, the normalized zero-counting measure $\rho_n$ of the
multiple Meixner-Pollaczek polynomial $Q_{n,n}$ in
Section~\ref{subsection:MMPP}, see \eqref{def:rho}, converges weakly to the
average $\nu_1=\int_{0}^1\mu_1^s\, \ud s$ of the measures $\mu_1^s$ in
\eqref{def:measures0}.
\end{proposition}

\begin{proof} This is a consequence of Theorem~\ref{theorem:locallyToepeig}.
The assumptions in the latter theorem are indeed satisfied: the interlacing
condition follows from Lemma~\ref{lemma:interlacing}, and the fact that
$\Gamma_1(s)\subset\er$ follows from Proposition \ref{monotonoussupports}.
\end{proof}

In addition to the set $\Gamma_1(s)$ in \eqref{def:Gamma0}, we define
\begin{equation}\label{def:Gamma2}
\Gamma_2(s)=\{ x\in\cee\mid |z_2(x,s)|=|z_3(x,s)|\}.
\end{equation}
We will now characterize the limiting zero distribution $\nu_1$ in terms of a
vector equilibrium problem from logarithmic potential theory
\cite{NS,SaffTotik}. Recall that for a pair of Borel measures $(\mu,\nu)$
supported in the complex plane, the \emph{mixed logarithmic energy} of $\mu$
and $\nu$ is defined as \cite{SaffTotik}
\begin{equation*}
I(\mu,\nu):=\iint \log\frac{1}{ |x-y|}\,\ud \mu(x)\, \ud \nu(y).
\end{equation*}

\begin{theorem}\label{theorem:equilibriumproblem} (Equilibrium problem)
 Suppose $t_1=-t_2=t\in (0,\pi/2)$. Then the asymptotic zero distribution, $\nu_1$, of
 $Q_{n,n}$  is the first component of the unique minimizer $(\nu_1,\nu_2)$ of the energy
 functional
\begin{equation*}
E(\mu,\nu):=I(\mu,\mu)+I(\nu,\nu)-I(\mu,\nu)+\int (\pi-2t)|x|\, \ud \mu(x),
\end{equation*}
among all vectors $(\mu,\nu)$ of positive measures such that
$\operatorname{supp} \mu \subset \mathbb{R}$, $\int\, \ud \mu=1$,
$\operatorname{supp} \nu \subset i\mathbb{R}$, $\int\, \ud \nu=1/2$
and $\nu$ is absolutely continuous with density satisfying
\begin{equation*}\label{constraintnu2}
\frac{\ud \nu(ix)}{|\ud x|}\leq\frac{2t}{\pi}.
\end{equation*}

The measures have the properties
\begin{equation*}\operatorname{supp} \nu_1=[-c_1,c_1]
\end{equation*}
and
 \begin{equation*}\operatorname{supp} (\sigma-\nu_2) =i\mathbb{R}\setminus
 (-ic_2,ic_2),
 \end{equation*}
where $\sigma$ is the positive, absolutely continuous measure on $i\mathbb{R}$
with constant density $\frac{2t}{\pi}$, and where $c_1$ and $c_2$ are positive
constants given by
 \begin{equation}\label{supportnu1}
c_1=\left(\frac{27b^4+18b^2-1+\sqrt{b^2+1}(9b^2+1)^{3/2}}{32b^2}\right)^{1/2}
\end{equation}
and
\begin{equation}\label{constraintactive}
c_2=\left(\frac{\sqrt{b^2+1}(9b^2+1)^{3/2}-27b^4-18b^2+1}{32b^2}\right)^{1/2},
\end{equation}
with $$b=\tan t.$$ Furthermore, $\nu_1$ is absolutely continuous with respect
to Lebesgue measure and has density
\begin{equation}\label{densityrel1bis}
\frac{\ud \nu_1}{\ud x}=\frac{1}{2\pi}\log\left|\frac{1+iw(x)}{1-iw(x)}\right|,
\qquad x\in [-c_1,c_1],
\end{equation}
where
\begin{equation}\label{densityrel2}
w(x)=\frac{(4+z(b^2-1))(4+z(b^2+1))}{16 z |x|}
\end{equation}
and $z=z(x)$ is the complex solution to the algebraic equation
\begin{equation}\label{densityrel3}
\frac{(4+z(b^2+1))^2
(4+z b^2 )}{64z}=x^2
\end{equation}
such that $\operatorname{Im}(w(x))<0$.
\end{theorem}

\begin{figure}[tbp]
\begin{center}
\subfigure{\label{fig:MP1}}\includegraphics[scale=0.3]{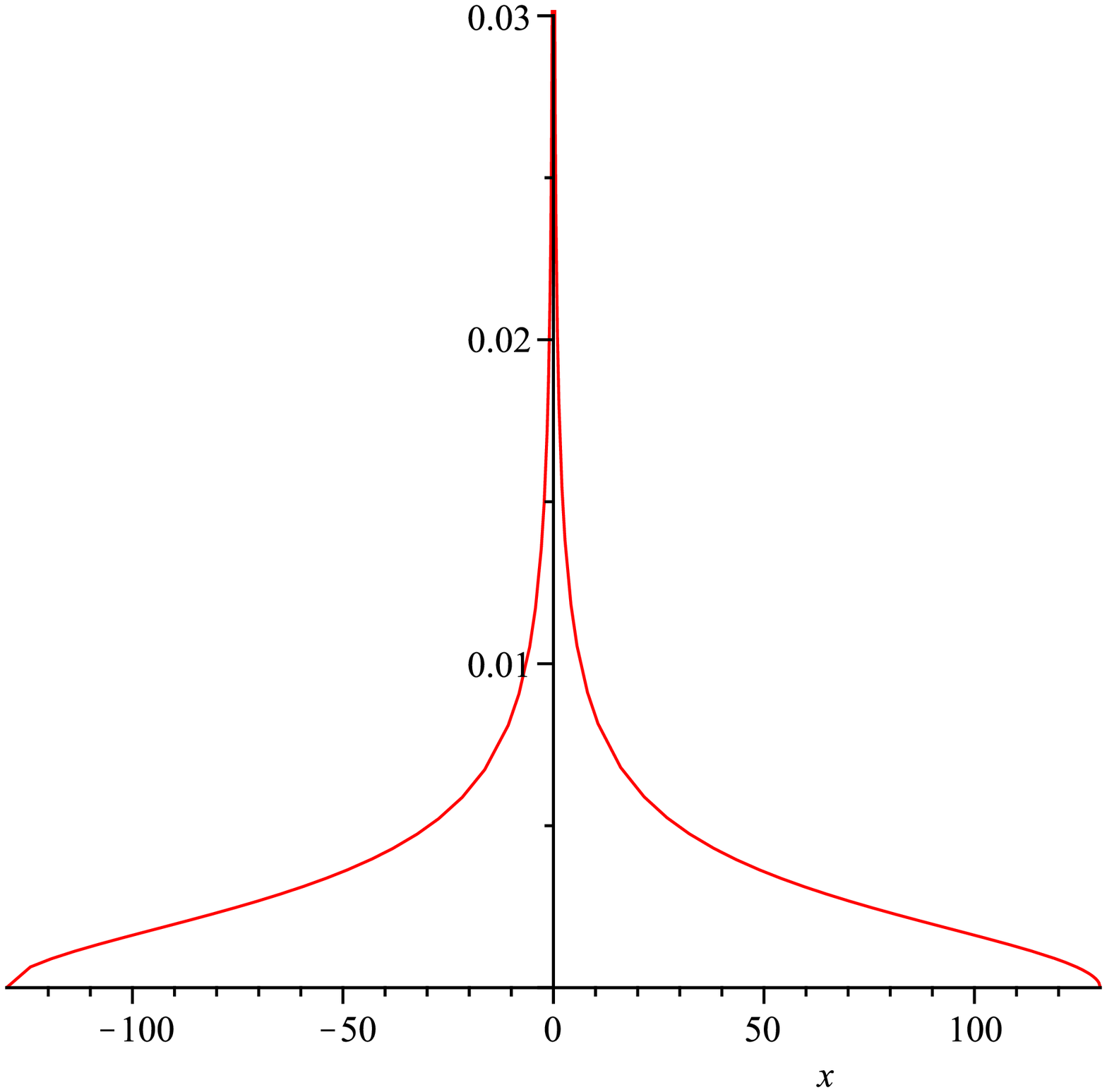}\hspace{5mm}
\subfigure{\label{fig:MP2}}\includegraphics[scale=0.3]{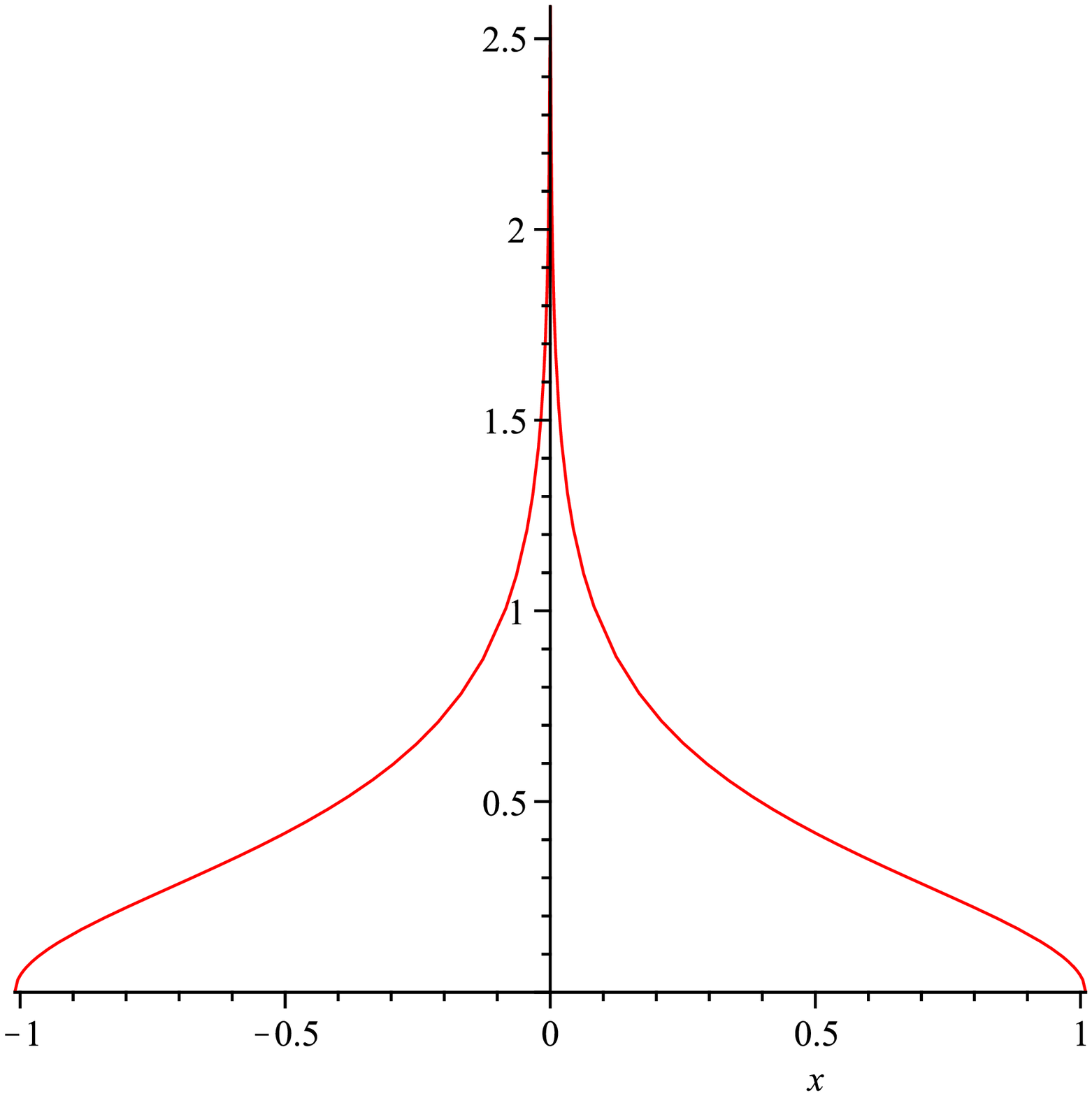}
\end{center}
\caption{Density of the measure $\nu_1$ for (a) $b=\tan t=100$ and (b) $b=\tan
t=0.1$.} \label{fig:MP}
\end{figure}

\begin{remark}
In general, there is no closed expression for the density of $\nu_1$, but in
the limiting case $t\to 0$, corresponding to the ordinary Meixner-Pollaczek
orthogonal polynomials, (\ref{densityrel3}) can be explicitly solved,
$z=8x^2-4+8i|x|\sqrt{1-x^2}$, giving $|z|=4$. Thus (\ref{densityrel2}) becomes
\begin{equation*}
w(x)=\frac{16-z^2}{16 z |x|}=\frac{16\overline{z}-|z|^2z}{16 |z|^2 |x|}=-i\frac{\operatorname{Im}(z)}{8|x|}=-i\sqrt{1-x^2},
\end{equation*}
and the density (\ref{densityrel3}) takes the form
\begin{align}\label{density3}
\frac{\ud \nu_1}{\ud x}
=\frac{1}{2\pi}\log\left(\frac{1+\sqrt{1-x^2}}{1-\sqrt{1-x^2}}\right)\chi_{\{|x|\leq 1\}}.
\end{align}
For non-zero values of $t\in(0,\pi/2)$, the density of $\nu_1$ turns out to
have the same qualitative features as \eqref{density3}; see Figure~\ref{fig:MP}
for some illustrations.

\end{remark}

Theorem~\ref{theorem:equilibriumproblem} will be proven in
Section~\ref{section:equilibrium}. The proof uses a general result for block
Toeplitz matrices \cite{Delvaux}, see also \cite{Duits}, together with some
explicit calculations using \eqref{algebraicequation:P}.

\subsection{Motivation: The six-vertex model with domain wall boundary conditions}
\label{section:sixvertex}

Multiple Meixner-Pollaczek polynomials appear in the study of the six-vertex
model in statistical mechanics, as we explain now.

Consider an $N\times N$ square lattice in the plane.
A configuration of the six-vertex model is an assignment of an
orientation to the edges of the lattice in such a way that each vertex is surrounded by
precisely two incoming and two outgoing edges. See Fig.~\ref{fig:grid5x5} for a
configuration with $N=5$. The name \emph{six-vertex model} refers to the fact that the
local behavior near each vertex is given by six possible edge configurations
(see Fig.~\ref{fig:6vertices}).

We consider the six-vertex model with domain wall boundary conditions
(DWBC). This means that the edges at the top and bottom of the lattice must be directed
outwards and those at the left and right of the lattice must be directed
inwards. 

\begin{figure}[h]
    \vspace{-10mm}

   \setlength{\unitlength}{0.1cm}
   \qquad\quad\,\,\,\begin{picture}(80,50)(-20,20)
       \put(10,10){\thicklines\circle*{1}}
       \put(10,20){\thicklines\circle*{1}}
       \put(10,30){\thicklines\circle*{1}}
       \put(10,40){\thicklines\circle*{1}}
       \put(10,50){\thicklines\circle*{1}}
       \put(20,10){\thicklines\circle*{1}}
       \put(20,20){\thicklines\circle*{1}}
       \put(20,30){\thicklines\circle*{1}}
       \put(20,40){\thicklines\circle*{1}}
       \put(20,50){\thicklines\circle*{1}}
       \put(30,10){\thicklines\circle*{1}}
       \put(30,20){\thicklines\circle*{1}}
       \put(30,30){\thicklines\circle*{1}}
       \put(30,40){\thicklines\circle*{1}}
       \put(30,50){\thicklines\circle*{1}}
       \put(40,10){\thicklines\circle*{1}}
       \put(40,20){\thicklines\circle*{1}}
       \put(40,30){\thicklines\circle*{1}}
       \put(40,40){\thicklines\circle*{1}}
       \put(40,50){\thicklines\circle*{1}}
       \put(50,10){\thicklines\circle*{1}}
       \put(50,20){\thicklines\circle*{1}}
       \put(50,30){\thicklines\circle*{1}}
       \put(50,40){\thicklines\circle*{1}}
       \put(50,50){\thicklines\circle*{1}}

       \put(1,10){\line(1,0){58}}
       \put(1,20){\line(1,0){58}}
       \put(1,30){\line(1,0){58}}
       \put(1,40){\line(1,0){58}}
       \put(,50){\line(1,0){58}}
       \put(10,1){\line(0,1){58}}
       \put(20,1){\line(0,1){58}}
       \put(30,1){\line(0,1){58}}
       \put(40,1){\line(0,1){58}}
       \put(50,1){\line(0,1){58}}
       \put(5,10){\thicklines{\vector(1,0){1}}}
       \put(5,20){\thicklines{\vector(1,0){1}}}
       \put(5,30){\thicklines{\vector(1,0){1}}}
       \put(5,40){\thicklines{\vector(1,0){1}}}
       \put(5,50){\thicklines{\vector(1,0){1}}}

       \put(55,10){\thicklines{\vector(-1,0){1}}}
       \put(55,20){\thicklines{\vector(-1,0){1}}}
       \put(55,30){\thicklines{\vector(-1,0){1}}}
       \put(55,40){\thicklines{\vector(-1,0){1}}}
       \put(55,50){\thicklines{\vector(-1,0){1}}}
       \put(10,5){\thicklines{\vector(0,-1){1}}}
       \put(20,5){\thicklines{\vector(0,-1){1}}}
       \put(30,5){\thicklines{\vector(0,-1){1}}}
       \put(40,5){\thicklines{\vector(0,-1){1}}}
       \put(50,5){\thicklines{\vector(0,-1){1}}}

       \put(10,55){\thicklines{\vector(0,1){1}}}
       \put(20,55){\thicklines{\vector(0,1){1}}}
       \put(30,55){\thicklines{\vector(0,1){1}}}
       \put(40,55){\thicklines{\vector(0,1){1}}}
       \put(50,55){\thicklines{\vector(0,1){1}}}

       \put(15,10){\thicklines\vector(1,0){1}}
       \put(15,20){\thicklines\vector(1,0){1}}
       \put(15,30){\thicklines\vector(1,0){1}}
       \put(15,40){\thicklines\vector(-1,0){1}}
       \put(15,50){\thicklines\vector(1,0){1}}
       \put(25,10){\thicklines\vector(1,0){1}}
       \put(25,20){\thicklines\vector(1,0){1}}
       \put(25,30){\thicklines\vector(-1,0){1}}
       \put(25,40){\thicklines\vector(1,0){1}}
       \put(25,50){\thicklines\vector(-1,0){1}}
       \put(35,10){\thicklines\vector(-1,0){1}}
       \put(35,20){\thicklines\vector(1,0){1}}
       \put(35,30){\thicklines\vector(-1,0){1}}
       \put(35,40){\thicklines\vector(1,0){1}}
       \put(35,50){\thicklines\vector(-1,0){1}}
       \put(45,10){\thicklines\vector(-1,0){1}}
       \put(45,20){\thicklines\vector(-1,0){1}}
       \put(45,30){\thicklines\vector(1,0){1}}
       \put(45,40){\thicklines\vector(-1,0){1}}
       \put(45,50){\thicklines\vector(-1,0){1}}

       \put(10,15){\thicklines\vector(0,-1){1}}
       \put(20,15){\thicklines\vector(0,-1){1}}
       \put(30,15){\thicklines\vector(0,1){1}}
       \put(40,15){\thicklines\vector(0,-1){1}}
       \put(50,15){\thicklines\vector(0,-1){1}}
       \put(10,25){\thicklines\vector(0,-1){1}}
       \put(20,25){\thicklines\vector(0,-1){1}}
       \put(30,25){\thicklines\vector(0,1){1}}
       \put(40,25){\thicklines\vector(0,1){1}}
       \put(50,25){\thicklines\vector(0,-1){1}}
       \put(10,35){\thicklines\vector(0,-1){1}}
       \put(20,35){\thicklines\vector(0,1){1}}
       \put(30,35){\thicklines\vector(0,1){1}}
       \put(40,35){\thicklines\vector(0,-1){1}}
       \put(50,35){\thicklines\vector(0,1){1}}
       \put(10,45){\thicklines\vector(0,1){1}}
       \put(20,45){\thicklines\vector(0,-1){1}}
       \put(30,45){\thicklines\vector(0,1){1}}
       \put(40,45){\thicklines\vector(0,1){1}}
       \put(50,45){\thicklines\vector(0,1){1}}

       \put(60,49){$x_1$}
       \put(60,39){$x_2$}
       \put(60,29){$x_3$}
       \put(60,19){$x_4$}
       \put(60,9){$x_5$}
       \put(48,-1){$y_5$}
       \put(38,-1){$y_4$}
       \put(28,-1){$y_3$}
       \put(18,-1){$y_2$}
       \put(8,-1){$y_1$}
    \end{picture}
    \vspace{20mm}
    \caption{A configuration of the six-vertex model with DWBC for $N=5$}
    \label{fig:grid5x5}
\end{figure}
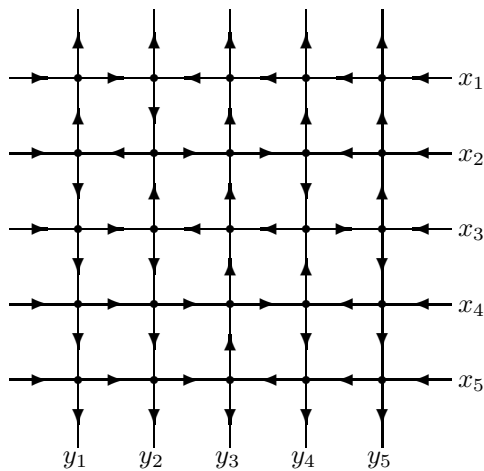

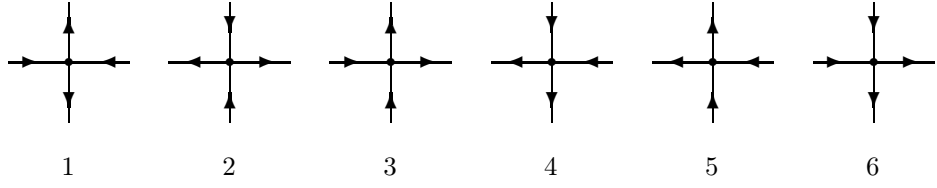
\begin{figure}[h]
\begin{center}
    \setlength{\unitlength}{1truemm}

   \begin{picture}(20,20)(0,0)
       \put(10,10){\thicklines\circle*{1}}
       \put(10,10){\line(1,0){8}}
       \put(10,10){\line(-1,0){8}}
       \put(10,10){\line(0,1){8}}
       \put(10,10){\line(0,-1){8}}
       \put(5,10){\thicklines\vector(1,0){1}}
       \put(15,10){\thicklines\vector(-1,0){1}}
       \put(10,5){\thicklines\vector(0,-1){1}}
       \put(10,15){\thicklines\vector(0,1){1}}
       \put(9,-5){1}
    \end{picture}
    \begin{picture}(20,20)(0,0)
       \put(10,10){\thicklines\circle*{1}}
       \put(10,10){\line(1,0){8}}
       \put(10,10){\line(-1,0){8}}
       \put(10,10){\line(0,1){8}}
       \put(10,10){\line(0,-1){8}}
       \put(5,10){\thicklines\vector(-1,0){1}}
       \put(15,10){\thicklines\vector(1,0){1}}
       \put(10,5){\thicklines\vector(0,1){1}}
       \put(10,15){\thicklines\vector(0,-1){1}}
       \put(9,-5){2}
    \end{picture}
    \begin{picture}(20,20)(0,0)
       \put(10,10){\thicklines\circle*{1}}
       \put(10,10){\line(1,0){8}}
       \put(10,10){\line(-1,0){8}}
       \put(10,10){\line(0,1){8}}
       \put(10,10){\line(0,-1){8}}
       \put(5,10){\thicklines\vector(1,0){1}}
       \put(15,10){\thicklines\vector(1,0){1}}
       \put(10,5){\thicklines\vector(0,1){1}}
       \put(10,15){\thicklines\vector(0,1){1}}
       \put(9,-5){3}
   \end{picture}
    \begin{picture}(20,20)(0,0)
       \put(10,10){\thicklines\circle*{1}}
       \put(10,10){\line(1,0){8}}
       \put(10,10){\line(-1,0){8}}
       \put(10,10){\line(0,1){8}}
       \put(10,10){\line(0,-1){8}}
       \put(5,10){\thicklines\vector(-1,0){1}}
       \put(15,10){\thicklines\vector(-1,0){1}}
       \put(10,5){\thicklines\vector(0,-1){1}}
       \put(10,15){\thicklines\vector(0,-1){1}}
       \put(9,-5){4}
    \end{picture}
\begin{picture}(20,20)(0,0)
       \put(10,10){\thicklines\circle*{1}}
       \put(10,10){\line(1,0){8}}
       \put(10,10){\line(-1,0){8}}
       \put(10,10){\line(0,1){8}}
       \put(10,10){\line(0,-1){8}}
       \put(5,10){\thicklines\vector(-1,0){1}}
       \put(15,10){\thicklines\vector(-1,0){1}}
       \put(10,5){\thicklines\vector(0,1){1}}
       \put(10,15){\thicklines\vector(0,1){1}}
       \put(9,-5){5}
    \end{picture}
\begin{picture}(10,10)(0,0)
       \put(10,10){\thicklines\circle*{1}}
       \put(10,10){\line(1,0){8}}
       \put(10,10){\line(-1,0){8}}
       \put(10,10){\line(0,1){8}}
       \put(10,10){\line(0,-1){8}}
       \put(5,10){\thicklines\vector(1,0){1}}
       \put(15,10){\thicklines\vector(1,0){1}}
       \put(10,5){\thicklines\vector(0,-1){1}}
       \put(10,15){\thicklines\vector(0,-1){1}}
       \put(9,-5){6}
    \end{picture}
\end{center}
\caption{Six types of vertices.} \label{fig:6vertices}
\end{figure}

To each of the $N$ rows of the lattice we associate a parameter $x_i\in\er$ and
similarly to each column a parameter $y_j\in\er$, $i,j=1,\ldots,N$. We also fix
a positive parameter $\gamma$ and we assume that $|x_i-y_j|<\gamma$ for all $i$
and $j$. We define the weight of the vertex in row $i$ and column $j$ according to its type as
$\sin(2\gamma)$ (type 1 or 2), $\sin(\gamma-(x_i-y_j))$ (type 3 or 4), or
$\sin(\gamma+(x_i-y_j))$ (type 5 or 6). 
Note that the weights are parameterized according to the so-called
disordered phase convention. 

The weight of a configuration is defined as the product of the weights
of all the $N\times N$ vertices in the configuration. The partition
function $Z_N = Z_N(x_1,\ldots,x_N,y_1,\ldots,y_N,\gamma)$ is defined as the
sum of the weights of all the consistent configurations of the $N\times N$
six-vertex model with DWBC. An explicit expression for the partition function
in terms of an $N\times N$ determinant was found by Izergin and Korepin
\cite{Iz,ICK}, see also \cite{Bax,KZ,Kup}.

\subsubsection*{The homogeneous case}

In the case where $x_i-y_j\equiv t$ for all $i,j=1,\ldots,N$,
for some fixed parameter $t\in(-\gamma, \gamma)$, the Izergin-Korepin formula reduces to
\begin{equation}\label{Z:homogenous:step1}
Z_N = Z_N(\gamma,t) = \frac{[\sin(\gamma+t)\sin(\gamma-t)]^{N^2}}{\left(
\displaystyle \prod_{n=0}^{N-1} n! \right)^2}\det M,
\end{equation}
where the matrix $M = (m_{i,j})_{i,j=1}^N$ has entries
\begin{equation}\label{Hankelmatrix}
m_{i,j} = \int_{-\infty}^{\infty} x^{i+j-2} e^{tx}w(x)\ dx
\end{equation}
with
\begin{equation}\label{def:mx} w(x) =
\frac{\sinh(\frac{x}{2}(\pi-2\gamma))}{\sinh(\frac{x}{2}\pi)},
\end{equation}
see e.g.~\cite{BF,ICK,Zinn}. The matrix $M$ is then precisely the \emph{moment
matrix} corresponding to the weight function $e^{tx}w(x)$ on the real line.
Standard considerations (e.g.\ \cite{Sz}) show that $\det M$ can be expressed
in terms of the monic orthogonal polynomials $P_n(x)$ defined by
$$ P_n(x) = x^n+\mathcal{O}(x^{n-1})
$$
for all $n$ and \begin{equation}\label{def:hn} \int_{-\infty}^{\infty}
P_n(x)x^m e^{tx}w(x)\ \ud x = h_n\delta_{m,n}
\end{equation}
for all $n,m$ with $m\leq n$. In fact, $\det M$ is expressed in terms of the
numbers $h_n$ in \eqref{def:hn} through the formula
\begin{equation}\label{moment:determinant} \det M = \prod_{n=0}^{N-1} h_n.
\end{equation}

Special choices of parameters lead to known families of
orthogonal polynomials. Indeed, Colomo and Pronko \cite{CP1,CP2}
showed that the Continuous Hahn, Meixner-Pollaczek and continuous Dual Hahn polynomials appear in this way.
In more general cases, the expressions \eqref{Z:homogenous:step1}--\eqref{moment:determinant} were
used to compute the asymptotics
of the partition function $Z_N$ for large $N$ in great
detail by means of the Riemann-Hilbert method~\cite{BF}, see also
\cite{BL1,BL2}.

\subsubsection*{The inhomogeneous case}

The situation in this paper corresponds to the case where
\begin{equation}\label{def:tj}x_i-y_j \equiv \left\{\begin{array}{ll} t_1, & 1\leq i\leq
n_1,\\ t_2, & n_1+1\leq i\leq n_1+n_2=N,\end{array}\right.\end{equation} for
some $t_1\neq t_2$ and all $j=1,\ldots,N$. Following the reasoning in
\cite{ICK}, see also the appendix in \cite{CP3}, 
one sees that the Izergin-Korepin formula reduces to
\begin{equation}\label{Z:blockhomogenous}
Z_N =
\frac{[\sin(\gamma+t_1)\sin(\gamma-t_1)]^{n_1N}[\sin(\gamma+t_2)\sin(\gamma-t_2)]^{n_2N}}{\left(\displaystyle \prod_{n=0}^{n_1-1}
n! \right) \left( \displaystyle \prod_{n=0}^{n_2-1} n!\right) \left(\displaystyle \prod_{n=0}^{N-1} n!\right) }\det M
\end{equation}
where the matrix $M= (m_{i,j})_{i,j=1}^N$ now has entries
\begin{equation}\label{def:M:inh} m_{i,j} = \left\{\begin{array}{ll}
\displaystyle \int_{-\infty}^{\infty} x^{i+j-2}e^{t_1x}w(x)\ud x, & 1\leq i\leq n_1,\\
\displaystyle \int_{-\infty}^{\infty} x^{i+j-n_1-2}e^{t_2x}w(x) \ud x, & n_1+1\leq i\leq
N,\end{array}\right.
\end{equation}
with $w$ still given by (\ref{def:mx}).
Thus $M$ is a moment matrix with respect to the system of weight functions
$e^{t_1 x}w(x)$ and $e^{t_2 x}w(x)$ on the real line.

The inhomogeneous model \eqref{Z:blockhomogenous} was studied in \cite{CP3, CPZ}
in connection with the calculation of the arctic curve. Leading order asymptotics 
of the partition function for the case $n_1 = 1$ in \eqref{Z:blockhomogenous}
was computed in \cite{CP3} for the disordered regime, and in \cite{CPZ}
for the anti-ferroelectric regime. The analysis is valid in fact 
for any $n_1$ as long as $n_1 = o(N)$ as $N \to \infty$.

It turns out that $\det M$ can be expressed in terms of monic \emph{multiple}
orthogonal polynomials $P_{k_1,k_2}(x)$ with respect to the system of weight
functions $e^{t_1 x}w(x)$ and $e^{t_2 x}w(x)$. The polynomial $P_{k_1,k_2}(x)$
is defined for any nonnegative integers $k_1,k_2$ by
$$ P_{k_1,k_2}(x) = x^{k_1+k_2}+\mathcal{O}(x^{k_1+k_2-1})
$$
and
\begin{equation}\label{def:hn:block} \int_{-\infty}^ {\infty}P_{k_1,k_2}(x)x^m e^{t_j x}w(x)\
\ud x = h_{k_1,k_2}^{(j)} \delta_{m,k_j},\quad m=0,\ldots,k_j,\quad j=1,2.
\end{equation}
We now have the following generalization of the formula
\eqref{moment:determinant}.

\begin{proposition} (Partition function)
Let $\gamma>0$, $t_1,t_2\in (-\gamma,\gamma)$, $t_1\neq t_2$ and recall the notation $w(x)$
in \eqref{def:mx} and $h_{k_1,k_2}^{(j)}$ in \eqref{def:hn:block}. Then the
moment matrix $M$ in \eqref{def:M:inh} has determinant
\begin{equation}\label{detM:block} \det M = \prod_{n=0}^{N-1}
h^{(j(n))}_{k_1(n),k_2(n)},
\end{equation}
where $\left(j(n)\right)_{n=0}^{N-1}$ is any sequence of $1$'s and $2$'s such
that
$1$ appears $n_1$ times and $2$ appears $n_2$ times. 
and where the associated sequences $(k_i(n)))_{n=0}^{N-1}$, $i=1,2$, are
defined recursively by $k_i(0)=0$ and
\begin{align*}k_i(n)-k_i(n-1)=\left\{\begin{array}{ll}1 &\quad \textrm{ if $j(n)=i$},\\
0, &\quad \textrm{ otherwise},\end{array}\right.\end{align*} for any
$n=1,\ldots,N-1$ and $i=1,2$.
\end{proposition}

\begin{proof}
We will give the proof for the particular sequence
$\left(j(n)\right)_{n=0}^{N-1}=( \underset{n_1 \textrm{
times}}{\underbrace{1,\ldots,1}}, \underset{n_2 \textrm{
times}}{\underbrace{2,\ldots,2}})$; it will be straightforward to extend the
proof to the more general sequences $\left(j(n)\right)_{n=0}^{N-1}$ in the
statement of the proposition. From the definition \eqref{def:M:inh} it follows
that
$$
M = \left(\langle  f_i,g_j \rangle \right)_{i,j=0}^{N-1}
$$
where we define the inner product
$$ \langle f,g\rangle = \int_{-\infty}^{\infty} f(x)g(x)\ \ud x
$$
and where we use the functions
\begin{align}
    \nonumber f_i(x) & = x^{i} e^{t_1 x}w(x), \qquad i = 0, \ldots, n_1-1, \\
    \nonumber f_{n_1+i}(x) & = x^{i} e^{t_2 x}w(x), \qquad i = 0, \ldots, n_2-1,
\end{align}
and
\begin{align}
    \nonumber g_n(x) & = x^{n}, \qquad n = 0, \ldots, N-1.
\end{align}
Define functions $\varphi_i(x)$ and $\psi_j(x)$, $i,j=0,\ldots,N-1$, by
bi-orthogonalizing the functions $f_i(x)$ and $g_j(x)$ in the following way:
$$ \varphi_i(x) = f_i(x)+\sum_{k=1}^{i-1} c_{k,i}f_k(x),
$$
$$ \psi_j(x) = g_j(x)+\sum_{k=1}^{j-1} b_{k,j}g_i(x),
$$
for appropriate coefficients $b_{k,j}$ and $c_{k,i}$, subject to the
orthogonality relations
\begin{equation}\label{def:hn:block:bis} \langle \varphi_i,\psi_j\rangle
= h_i\delta_{i,j}.
\end{equation}
It is not hard to see that we can identify $\psi_n(x) = P_{k_1(n),k_2(n)}(x)$
and $h_n=h^{(j(n))}_{k_1(n),k_2(n)}$.

Then we have that
$$ \det M \equiv \left(\langle  f_i,g_j \rangle \right)_{i,j=0}^{N-1} =
\left(\langle  \varphi_i,\psi_j \rangle \right)_{i,j=0}^{N-1} =
\prod_{n=0}^{N-1} h_n,
$$
where 
the last step follows in a trivial way from \eqref{def:hn:block:bis}. In view
of the identifications mentioned in the previous paragraph, we then obtain
\eqref{detM:block}.
\end{proof}

It is straightforward to generalize the above reasoning to the case of multiple
values of the differences in \eqref{def:tj}. In general, one could
even allow \emph{both}  $x_i$
and $y_j$ to take multiple values and then one should deal with
\emph{multiple orthogonal polynomials of mixed type}. We leave the details to the interested reader.

\subsubsection*{The free fermion line: Meixner-Pollaczek weights}

The value $\gamma= \pi/4$ corresponds to the so-called free fermion line.
As first observed in \cite{CP2}, in this case the above weight functions are related
to the Meixner-Pollaczek weight.
Indeed, we then have
\begin{equation}\label{weight:freefermion}
e^{t_j x}w(x)  = e^{t_j x}\frac{\sinh(\frac{x\pi}{4})}{\sinh(\frac{x\pi}{2})} =
e^{t_j x}\frac{1}{2\cosh(\frac{x\pi}{4})}.
\end{equation}
We may compare this with the classical Meixner-Pollaczek weight function
\cite{Koekoek},
\begin{equation}\label{weight:meixnerpollaczek}
\frac{1}{2\pi} e^{2tx}\Gamma(x+ix)\Gamma(\lam-ix).
\end{equation}
By invoking the identity
$$\Gamma(z)\Gamma(1-z) = \frac{\pi}{\sin(\pi z)},$$
we see that \eqref{weight:meixnerpollaczek} for $\lam=1/2$ and $t=2t_j$ reduces
to
\begin{equation}\label{weight:meixnerpollaczek:bis} \frac{1}{2\pi} e^{4t_j x}\Gamma(1/2+ix)\Gamma(1/2-ix) = e^{4t_j x}\frac{1}{2\cosh(\pi x)}.
\end{equation}
Thus the weight functions \eqref{weight:freefermion} and
\eqref{weight:meixnerpollaczek:bis} are the same up to a scaling of the
variable $x$ by a factor $4$.

It is an easy job to evaluate the Izergin-Korepin formula for the partition
function explicitly on the free fermion line, so our results will not lead to
new insights in that perspective. They should rather be considered as giving exact and
asymptotic information on the multiple Meixner-Pollaczek polynomials in its own
right.

\subsection{Outline of the paper}

The rest of this paper is organized as follows: In Section
\ref{section:recurrence} we consider multiple Meixner-Pollaczek polynomials
with respect to two general weights and prove Theorem
\ref{theorem:recurrenceMMPP}. In Section \ref{section:zeroasymptotics} we
establish Theorem~\ref{theorem:locallyToepeig} and
Lemma~\ref{lemma:interlacing}, which together lead to Proposition
\ref{theorem:zeroseigenvalues}. In Section \ref{section:equilibrium} we prove
Theorem \ref{theorem:equilibriumproblem} using the theory of eigenvalue
asymptotics for banded block Toeplitz matrices.

\section{
Proofs of Lemma~\ref{lemma:interlacing} and Theorem~\ref{theorem:recurrenceMMPP}}
\label{section:recurrence}
\subsection{Proof of Lemma~\ref{lemma:interlacing}}
\label{subsection:proofinterlacing}

In this section we prove Lemma~\ref{lemma:interlacing} on the fact that the
multiple Meixner-Pollaczek polynomials $P_{k_1,k_2}$ exist and are unique, and have real
and interlacing zeros. It is well known that the corresponding result about orthogonal polynomials with
respect to one weight function holds; we will rely on a generalization of this fact due to
Kershaw, \cite{Kershaw}. In our context it amounts to the statement that a
sufficient condition is that for any
non-negative integer $k=k_1+k_2$ and any polynomials $A$ and $B$  (not both identically
zero) of degrees at most $k_1$ and $k_2-1$ (or $k_1-1$ and $k_2$)
respectively, the function $f(x)=A(x)w_1(x)+B(x)w_2(x)$ has at most $k$
zeros. 
Since
\begin{equation*}A(x)w_1(x)+B(x)w_2(x)
=\frac{1}{2\pi}e^{2t_2x}\left|\Gamma\left(\lambda+ix\right)\right|^2
\left(A(x)e^{2(t_1-t_2)x}+B(x)\right),
\end{equation*}
the conclusion will certainly follow if we can show that, for any real $t$,
$g(x):=A(x)e^{2tx}+B(x)$ has at most $k$ zeros whenever $A$ and $B$ are
polynomials such that $\operatorname{deg}A+\operatorname{deg}B\leq k-1$. (By
convention, the zero polynomial has degree $-1$.) This can easily be shown
by induction, see e.g.\ \cite[p.~138]{NS}. $\bol$

\subsection{Some generalities}
For $j=1,2$, let $w_j$ be integrable real functions on the real line such that
the measures $w_j(x)\ud x$ have moments of all orders. Suppose that for any
non-negative integers $k_1$ and $k_2$  there exists a unique monic multiple
orthogonal polynomial $P_{k_1,k_2}$ with respect to the weights $w_{1}, w_{2}$,
that is, a polynomial of degree $k=k_1+k_2$ satisfying the orthogonality
conditions
\begin{equation*}
\int_{-\infty}^{\infty} P_{k_1,k_2}(x)x^m w_j(x)\,\ud x= 0,\quad \textrm{ for
}m=0,\ldots, k_j-1,\quad j=1,2.
\end{equation*}
Let $\gamma_{k_1,k_2}$ denote the sub-leading coefficient of $P_{k_1,k_2} $, so
that $P_{k_1,k_2}(x)=x^k+ \gamma_{k_1,k_2}x^{k-1}+\mathcal{O}(x^ {k-2})$.
For $j=1,2$, put
\begin{equation}\label{def:h} h_{k_1,k_2}^{(j)}:=\int_{-\infty}^{\infty} x^{k_j} P_{k_1,k_2}(x)w_j(x)\ud x,
\end{equation} the first
non-vanishing moments.

We begin by stating a general four term recurrence formula for multiple
orthogonal polynomials on the real line in terms of their sub-leading
coefficients and first non-vanishing moments. This standard fact is
shown in \cite{Mah}  and can also be derived from the Riemann-Hilbert problem
for multiple orthogonal polynomials, see \cite{Ger}.

\begin{proposition}\label{generalrecurrence}
For any positive integers $k_1$ and $k_2$, the multiple orthogonal polynomials
satisfy the following four term recurrence relation:
\begin{align}\label{recurrence1}
P_{k_1+1,k_2}(z)=&(z+\gamma_{k_1+1,k_2}-\gamma_{k_1,k_2})P_{k_1,k_2}(z)\nonumber\\
&-\left(\frac{h_{k_1,k_2}^{(1)}}{h_{k_1-1,k_2}^{(1)}}+\frac{h_{k_1,k_2}^{(2)}}
{h_{k_1,k_2-1}^{(2)}}\right)P_{k_1,k_2-1}(z)\nonumber\\
&-\frac{h_{k_1,k_2}^{(1)}}{h_{k_1-1,k_2-1}^{(1)}}P_{k_1-1,k_2-1}(z).
\end{align}
By symmetry between the two indices, a corresponding recurrence relation for
$P_{k_1,k_2+1}(z)$ is obtained by interchanging $k_1$ and $k_2$ and superindices.
\end{proposition}

\subsection{Rodrigues formula for multiple Meixner-Pollaczek polynomials}
\label{subsection:MMPPrec}

From standard results on the ordinary
Meixner-Pollaczek polynomials, (see e.g. equations (1.7.2) (1.7.4) in
\cite{Koekoek} with  $\phi=t_1+\pi/2$), we have the orthogonality relation
\begin{equation}\label{orthogonality}
\int_{-\infty}^{\infty} P_{m,0}(x)P_{n,0}(x)w_1(x)\,\ud x= \frac{n!\Gamma(n+2\lambda)}{(2\cos
t_1)^{2\lambda+2n}}\delta_{mn},
\end{equation}
and the recurrence relation
\begin{equation}\label{recurrenceorthogonal}
xP_{n,0}(x)=P_{n+1,0}(x)+(n+\lambda)(\tan t_1) P_{n,0}(x)+
\frac{n(n+2\lambda-1)} {4 \cos^2 \!t_1} P_{n-1,0}(x).
\end{equation}

For any real parameter $t$, define the finite difference operator $L_t$,
acting on functions $f:\mathbb{C}\to \mathbb{C}$, by the equation
\begin{align}
\left(L_t f\right)(x) &=e^{it} f(x+i/2)-e^{-it} f(x-i/2).
\end{align}

\begin{lemma}\label{commutation}
For any real $t_1$, $t_2$, the operators $L_{t_1}$ and $L_{t_2}$ commute.
\end{lemma}
\begin{proof} Straightforward calculation.
\end{proof}
\begin{lemma}\label{PI}
Let $t$ be a real number and $f$, $g$ analytic functions in a domain containing
the strip $\Omega=\{z:|\operatorname{Im} (z)|\leq 1/2\}$, and assume that there
are positive numbers $C$ and $\epsilon$ such that
 \begin{equation}\label{decay}|f(z)g(z)e^{2tz}| <C e^{-\epsilon|\operatorname{Re}(z)|}
 \end{equation} for all $z\in \Omega$. 
Then the following integration by parts formula holds:
\begin{equation}\label{IP}
\int_{-\infty}^{\infty} f(x)\left(L_t g\right)(x)e^{2tx}\,\ud x=-\int_{-\infty}^{\infty}
\left(L_0 f\right)(x) g(x)e^{2tx}\,\ud x.
\end{equation}

\end{lemma}
\begin{proof}
Using the definition of $L_t$, we can split the integral of the left hand side
into two terms and  shift the contours of integration from the real line
to $\mathbb{R}\pm i/2$ for the first and second terms, respectively, by Cauchy's theorem. This gives
\begin{align*}
\int_{-\infty}^{\infty} f(x)e^{2tx}\left(L_t g\right)(x)\,\ud x &=\int_{-\infty}^{\infty}
f(x)\left(e^{t(2x+i)} g(x+i/2)-e^{t(2x-i)} g(x-i/2)\right)\,\ud x\\
&=\int_{-\infty}^{\infty} \left(f(u-i/2)e^{2tu} g(u)-f(u+i/2)e^{2tu} g(u)\right)\,\ud u\\
&=-\int_{-\infty}^{\infty} \left(L_0 f\right)(u)e^{2tu} g(u)\,\ud u.
\end{align*}
\end{proof}


We can now derive a Rodrigues type formula for the multiple Meixner-Pollaczek
polynomials, which will be the tool to calculate explicit recurrence
coefficients.

\begin{proposition}\label{representation}
Let $k_1$ and $k_2$ be non-negative integers and put $k=k_1+k_2$.   Let
$L^m:= \underset{m \textrm{
times}}{\underbrace{L\circ\cdots\circ L}}$
denote the $m$th iterate of an operator $L$. Then, for any $t_1, t_2 \in (-\pi/2, \pi/2)$
with $t_1 \neq t_2$, and any $\lambda>0$, the multiple Meixner-Pollaczek polynomial $P_{k_1,k_2}$ satisfies the Rodrigues
formula
\begin{equation}
\left( L_{t_1}^{k_1} L_{t_2}^{k_2}\left(\left|\Gamma\left(\lambda+k/2+i
\ \cdot \ \right)\right|^2\right)\right)(x)
=c_{k_1,k_2}P_{k_1,k_2}(x)\left|\Gamma\left(\lambda+ix\right)\right|^2,
\end{equation}
where 
\begin{equation}c_{k_1,k_2}=\left(-2i\right)^{k}(\cos t_1)^{k_1} (\cos t_2)^{k_2}.
\end{equation}
\end{proposition}
\begin{proof}
Define the function $$f_m(x)=\left|\Gamma(\lambda+m/2+ix)\right|^2,$$ for any
non-negative integer $m$. First of all, we note that by the properties
$\Gamma(z+1)=z\Gamma(z)$ and $\Gamma(\overline{z})=\overline{\Gamma(z)}$ of the
gamma function, it follows immediately from the definitions that if $R_m(x)$ is
a polynomial of degree $m$ with leading coefficient $a_m$, then
\begin{align*}
\left(L_t \left(
R_m f_{1}\right)\right)(x)=&e^{it}R_m(x+i/2)\Gamma\left(\lambda+ix\right)\Gamma
\left(\lambda+1-ix\right)\\
&-e^{- it
}R_m(x-i/2)\Gamma\left(\lambda+1+ix\right)\Gamma\left(\lambda-ix\right)\\
=&\left(e^{it}\left(\lambda-ix\right)R_m(x+i/2)-e^{-
it}\left(\lambda+ix\right)R_m(x-i/2)\right)f_{0}(x)\\ =&R_{m+1}(x)f_{0}(x)
\end{align*}
where $R_{m+1}$ is a polynomial of degree $m+1$ with leading coefficient
\begin{equation}\label{Lofpolynomial}a_{m+1}=-2i a_m \cos t.\end{equation}
Equivalently, by simply replacing the parameter $\lambda$ by $\lambda+n/2$,
\begin{equation}\label{recursionlambda}
\left(L_t \left( R_m f_{n+1}\right)\right)(x)=R_{m+1}(x) f_{n}(x).
\end{equation}
By induction over $k$, it follows from
(\ref{recursionlambda}) and (\ref{Lofpolynomial}) that
\begin{equation*}
\left(L_t^k f_k\right)(x) =R_k(x)f_0(x),
\end{equation*}
where $R_k$ is a polynomial of degree $k$ in $x$ with leading
coefficient $\left(-2i\cos t \right)^k$. Therefore,
\begin{equation*}
\left(L_{t_1}^{k_1}L_{t_2}^{k_2}f_k\right)(x)
=c_{k_1,k_2}\tilde{P}_{k_1,k_2}(x)f_0(x),
\end{equation*}
for some monic $k$th degree polynomial $\tilde{P}_{k_1,k_2}$.

Using this representation we can check the orthogonality conditions. Let $m<
k_1$ be a non-negative integer. Note that the choices
$f(z)=\left|\Gamma\left(c+iz\right)\right|^2$ for any real $c>1/2$ and $g$ a
polynomial satisfy condition (\ref{decay}) of Lemma \ref{PI}; this can be seen
from the asymptotics of the Gamma function valid  as $|z|\to \infty$ with
$|\operatorname{Arg}(z)|$ <$\pi-\epsilon$,
$$\Gamma(z)=\sqrt{2\pi}{z}\left(\frac{z}{e}\right)^{z}(1+o(1))$$
(Stirling's formula).
By definition of the weight function and
applying Lemma \ref{PI} $k_1$ times, we get
\begin{align*}\int_{-\infty}^{\infty} \tilde{P}_{k_1,k_2}(x)x^m w_1(x)\,\ud x=
&\frac{1}{2\pi c_{k_1,k_2}}\int_{-\infty}^{\infty} x^m
\left(L_{t_1}^{k_1}L_{t_2}^{k_2} f_k\right)(x)e^{2t_1x}\,\ud x\\
=&\frac{(-1)^{k_1}}{2\pi c_{k_1,k_2}}\int_{-\infty}^{\infty} L_0^{k_1} (x^m)
\left(L_{t_2}^{k_2} f_k\right)(x)e^{2t_1x}\,\ud x=0,
\end{align*}
since $L_0$ acting on non-zero polynomials decreases their degree by one. By
Lemma \ref{commutation}, the same argument applies in checking the
orthogonality relations with respect to $w_2$.
\end{proof}

\subsection{Proof of Theorem~\ref{theorem:recurrenceMMPP}}
We will need explicit expressions for the first non-vanishing moments, defined by (\ref{def:h}). These
are readily calculated using Proposition \ref{representation}.
\begin{proposition} \label{moment} The first non-vanishing moments $h_{k_1,k_2}^{(j)}$ of the
multiple Meixner-Pollaczek polynomial  $P_{k_1,k_2}$ are given by
\begin{equation}\label{moment1}
h_{k_1,k_2}^{(1)}=\frac{\Gamma(2\lambda+k)k_1 !(\sin(t_1-t_2))^{k_2}}
{2^{2\lambda +k+k_1}(\cos t_1)^{k+k_1+2\lambda} (\cos t_2)^{k_2}}
\end{equation}
and
\begin{equation}\label{moment2}
h_{k_1,k_2}^{(2)}=\frac{\Gamma(2\lambda+k)k_2 !(\sin(t_2-t_1))^{k_1}}
{2^{2\lambda+k+k_2}(\cos t_2)^{k+k_2+2\lambda} (\cos t_1)^{k_1}}.
\end{equation}
\end{proposition}
\begin{proof}
Consider the case $j=1$; by Lemma \ref{commutation} the $j=2$ case is
completely analogous. Reasoning as in the proof of the orthogonality relations,
and noting that $L_0( x^k)$ is a polynomial of degree $k-1$ with leading
coefficient $i k$, and  that $L_0 (e^{xt})=2i e^{2xt}\sin t$, we find that
\begin{align*}
h_{k_1,k_2}^{(1)} =&\frac{(-1)^{k_1}}{2\pi c_{k_1,k_2}}\int_{-\infty}^{\infty}
L_0^{k_1} (x^{k_1}) e^{2t_1 x}\left(L_{t_2}^{k_2} f_k\right)(x)\,\ud x\\
=&\frac{(-1)^{k}i^{k_1}k_1 !}{2\pi c_{k_1,k_2}}\int_{-\infty}^{\infty}
\left(2i\sin(t_1-t_2)\right)^{k_2} e^{2x(t_1-t_2)}  e^{2t_2 x} f_k(x)\,\ud x\\
=&\frac{k_1 ! \Gamma(2\lambda+k)(\sin(t_1-t_2))^{k_2}} {2^{2\lambda+k+k_1}(\cos
t_1)^{k+k_1+2\lambda} (\cos t_2)^{k_2}}.
\end{align*}
Here we made use of the orthogonality relation \eqref{orthogonality} (with
$m=n=0$) to compute the integral.
\end{proof}

\begin{proposition}\label{leadingcoeffprop}
The sub-leading coefficient $\gamma_{k_1,k_2}$ of $P_{k_1,k_2}(x)$ is given by
\begin{equation}\label{leadingcoeff}
\gamma_{k_1,k_2}=-\frac{(2\lambda+k-1)}{2}(k_1\tan t_1+k_2\tan t_2).
\end{equation}
\end{proposition}
\begin{proof}
Let $k=k_1+k_2$ be fixed. The polynomial
$P_{k_1+1,k_2}-P_{k_1,k_2+1}$ is clearly of degree $k$ and satisfies
$k_j$ orthogonality conditions with respect to $w_j$, for $j=1,2$. It is thus a
multiple of  $P_{k_1,k_2}$, and reading off the leading coefficient gives
$P_{k_1+1,k_2}-P_{k_1,k_2+1}=(\gamma_{k_1+1,k_2}-\gamma_{k_1,k_2+1})P_{k_1,k_2}$.
Multiplying this relation by $x^{k_1}$ and integrating with respect to $w_1$
gives the equation
\begin{equation*}
0-h_{k_1,k_2+1}^{(1)}=(\gamma_{k_1+1,k_2}-\gamma_{k_1,k_2+1})h_{k_1,k_2}^{(1)},
\end{equation*}
which by Proposition \ref{moment} can be written
\begin{equation}\label{gammadifference}
\gamma_{k_1+1,k_2}-\gamma_{k_1,k_2+1}=\frac{(2\lambda+k)}{2}(\tan t_2-\tan
t_1).
\end{equation}
Identifying coefficients in the recurrence relation
(\ref{recurrenceorthogonal}) for the ordinary monic orthogonal
Meixner-Pollaczek polynomials $P_{k_1,0}$ with respect to $w_1$, gives
$\gamma_{k_1+1,0}=\gamma_{k_1,0}+a_{k_1}$ and so
\[ \gamma_{k_1+1,0}=-\sum_{j=0}^{k_1}a_{j}=-\frac{(k_1+1)}{2}(2\lambda+k_1)\tan t_1. \]
Then repeated application of (\ref{gammadifference}) leads to the claim,
for any $k_1+k_2=k$.
\end{proof}

Finally we are ready for the proof of Theorem \ref{theorem:recurrenceMMPP}.

\begin{proof}[Proof of Theorem \ref{theorem:recurrenceMMPP}]
With the explicit expressions for $\gamma_{k_1,k_2}$ and $h^{(j)}_{k_1,k_2}$
given by Propositions \ref{moment} and \ref{leadingcoeffprop}, this follows
from the general recurrence relation (\ref{recurrence1}).
\end{proof}

\section{
Proof of Theorem~\ref{theorem:locallyToepeig}}
\label{section:zeroasymptotics}


In this section we will prove the general Theorem~\ref{theorem:locallyToepeig}
on the asymptotic zero distribution of a sequence of polynomials $Q_{k,n}$
generated by a recurrence relation
\eqref{reccoef:Jacobi}--\eqref{reccoef:Toeplitz}. The main idea of the proof
follows the scalar case $r=1$ by Kuijlaars-Rom\'an \cite[Theorem 1.2]{Roman},
see also \cite{Coussement,KVA}. But we will need some nontrivial modifications
due to the fact that $r$ may be greater than $1$.

The main tool in the proof is the following result on ratio asymptotics for the
$Q_{k,n}$, compare with \cite[Lemma 2.2]{Roman}.

\begin{lemma}\label{lemma:Roman} (Ratio asymptotics) Under the assumptions of
Theorem~\ref{theorem:locallyToepeig}, we have that for each $s>0$ there exists
$R>0$ so that all zeros of $Q_{k,n}$ belong to $[-R,R]$ whenever $k\leq
(s+1)n$. Moreover,
\begin{equation}\label{ratioasym} \lim_{k/n\to s}
\frac{Q_{k,n}(x)}{Q_{k+r,n} (x)} = z_1(x,s),\end{equation}
uniformly on compact
subsets of $\cee \setminus [-R,R]$, where $z_1(x,s)$ is the solution to the algebraic equation (\ref{fzlambda0}) with smallest modulus.
\end{lemma}

\begin{proof}
The claim about the boundedness of the zeros of $Q_{k,n}$ follows in a rather
standard way from the assumptions, see e.g.\ \cite[Proof of Lemma 2.2]{Roman}.
Now we turn to the claim \eqref{ratioasym}. We consider the family of functions
\begin{equation}\label{normalfamily}
\mathcal H = \left\{\frac{Q_{k,n}(x)}{Q_{k+1,n}(x)}\mid  k,n\in\enn,k\leq
(s+1)n\right\}.
\end{equation}

From the assumption that the zeros of $(Q_{k,n})_k$ are real and interlacing,
it follows that $\mathcal H$ is a normal family (in the sense of Montel) on
$\bar{\cee} \setminus [-R,R]$, see e.g.\ \cite[Proof of Lemma 2.2]{Roman}.

Using induction on $l$, we will show the following.\smallskip

\textbf{Claim}: For any $m\in\{0,\ldots,r-1\}$ and for each $l\geq 1$, the
following holds. If $(k_i)_i$, $(n_i)_i$ are sequences of non-negative integers
with $k_i,n_i\to\infty$, $rk_i/n_i\to s$ as $i\to\infty$, so that
$$f(x) := \lim_{i\to\infty} \frac{Q_{rk_i+m,n_i}(x)}{Q_{rk_i+m+r,n_i} (x)}$$ exists for $|x| > R$, then
$$f(x) =z_1(x,s) (1+\mathcal{O}(x^{-l}))$$ as $x\to\infty$. \smallskip

Let us prove this claim. We have $z_1(x,s) = x^{-r}(1+ O(1/x))$ as $x\to\infty$
(Lemma~\ref{lemma:mu0prob}(a)), and so it is clear that the claim holds for $l
= 1$.

Now assume that the claim holds for $l\geq 1$. We will prove that it also holds
for $l+r$. We will prove this when $m=0$; the proof for the other values of $m$
can be given in a similar way. Letting $(k_i)_i$, $(n_i)_i$ be as in the claim,
our goal will be to prove that the function $\epsilon(x)$ defined by
\begin{equation}\label{def:eps}
\lim_{i\to\infty} \frac{Q_{rk_i,n_i}(x)}{Q_{rk_i+r,n_i} (x)}=z_1(x,s)
(1+\epsilon(x))^{-1},
\end{equation}
satisfies $\epsilon(x)=\mathcal{O}(x^{-l-r})$ for $x\to\infty$.

Let us prove this. Since $rk_i/n_i\to s$ as $i\to\infty$, we may assume that
$$rk_i \leq (s + 1)n_i-(r-1)$$ for every $i$. For $j = -(r-1),\ldots,r\beta$, we then
have that
$$ \frac{Q_{rk_i-j,n_i}}{Q_{rk_i+1-j,n_i}}$$ belongs to the family $\mathcal H$.
Since $\mathcal H$ is a normal family, we
may assume, by passing to a subsequence if necessary, that
$$ f^{(j)}(x) = \lim_{i\to\infty} \frac{Q_{rk_i-j,n_i}(x)}{Q_{rk_i+1-j,n_i}(x)}
$$
exists for $x\in\bar{\cee}\setminus [-R,R]$ and $j=-(r-1),\ldots,r\beta$.

Taking the $k_i$th block row in \eqref{reccoef:Jacobi} with $n=n_i$ and using
\eqref{J:part}, we obtain the matrix-vector relation
\begin{equation}\label{recursion:mxv} x\underset{r\times 1}{\underbrace{\mathbf
Q_{k_i,n_i}(x)}} = \underset{r\times r(\beta+2)}{\underbrace{\begin{pmatrix}
A_{k_i,n_i}^{(\beta)} & \cdots & A_{k_i,n_i}^{(-1)}
\end{pmatrix}}}
\underset{r(\beta+2)\times 1}{\underbrace{\begin{pmatrix}\mathbf Q_{k_i-\beta,n_i}(x)\\ \vdots \\
\mathbf Q_{k_i+1,n_i}(x)\end{pmatrix}}},
\end{equation}
where we denote with $\mathbf Q_{k,n}(x)$ the $r\times 1$ column vector
$$ \mathbf Q_{k,n}(x) := \begin{pmatrix}Q_{rk,n}(x)\\ \vdots \\
Q_{r(k+1)-1,n}(x)\end{pmatrix}.
$$
Dividing both sides of \eqref{recursion:mxv} by the scalar function
$Q_{rk_i,n_i}(x)$, and taking the limit $i\to\infty$, we find
\begin{equation}\label{recursion:mxv2} x\widetilde{\mathbf
Q}^{(0)}(x) = \begin{pmatrix} A_{s}^{(\beta)} & \cdots & A_{s}^{(-1)}
\end{pmatrix}
\begin{pmatrix}\widetilde{\mathbf Q}^{(\beta)}(x)\\ \vdots \\
\widetilde{\mathbf Q}^{(-1)}(x)\end{pmatrix},
\end{equation}
where we used \eqref{reccoef:Toeplitz} and where we set
\begin{equation}\label{def:Qtil} \widetilde{\mathbf{Q}}^{(j)}(x) :=
\lim_{i\to\infty} \frac{\mathbf Q_{k_i-j,n_i}(x)}{Q_{rk_i,n_i}(x)},\qquad
j=-1,\ldots,\beta,
\end{equation}
which exists entrywise due to our assumptions.
It will be convenient to rewrite $\widetilde{\mathbf{Q}}^{(j)}(x)$,
$j=0,\ldots,\beta$, as a telescoping product:
\begin{align}\label{telescoping}  \widetilde{\mathbf{Q}}^{(j)}(x) = & \lim_{i\to\infty}
\frac{\mathbf Q_{k_i-j,n_i}(x)}{Q_{rk_i,n_i}(x)} \nonumber \\
= &\lim_{i\to\infty}
\frac{\mathbf Q_{k_i-j,n_i}(x)}{\mathbf
Q_{k_i-j+1,n_i}(x)}\lim_{i\to\infty}\frac{\mathbf Q_{k_i-j+1,n_i}(x)}{\mathbf
Q_{k_i-j+2,n_i}(x)}\times \cdots
 \times\lim_{i\to\infty}\frac{\mathbf Q_{k_i,n_i}(x)}{Q_{rk_i,n_i}(x)},
\end{align}
where by abuse of notation we write $\frac{\mathbf a}{\mathbf b}$ and $\mathbf
a \mathbf b$ for two vectors $\mathbf a,\mathbf b$ of length $r$ to denote
their entrywise quotient and product respectively.
Each of the limits in \eqref{telescoping} exists again entrywise due to our
assumptions. Applying the induction hypothesis to each of the limits in the
telescoping product \eqref{telescoping}, we find that
\begin{equation}\label{def:Qtil2} \widetilde{\mathbf{Q}}^{(j)}(x) =
 z_1(x,s)^{j}(1+\mathcal O(x^{-l}))\widetilde{\mathbf{Q}}^{(0)}(x),\quad
x\to\infty,
\end{equation}
for any $j=0,\ldots,\beta$.

Next we turn to the term $A_s^{(-1)}\widetilde{\mathbf{Q}}^{(-1)}(x)$ in the
expansion of the right hand side of \eqref{recursion:mxv2}. Denoting $\mathbf
e_1:=(1,0,\ldots,0)^T$, we can write this as
\begin{align}\nonumber A_s^{(-1)}\widetilde{\mathbf{Q}}^{(-1)}(x) &= A_s^{(-1)} \lim_{i\to\infty}
\frac{\mathbf Q_{k_i+1,n_i}(x)}{Q_{rk_i,n_i}(x)}\\
\nonumber &=
\left(\lim_{i\to\infty}\frac{Q_{r(k_i+1),n_i}(x)}{Q_{rk_i,n_i}(x)}\right)
A_s^{(-1)}\mathbf e_1
\\ \nonumber &= z_1(x,s)^{-1}(1+\epsilon(x))A_s^{(-1)}\mathbf e_1
\\ \label{def:Qtil3} &= z_1(x,s)^{-1}(1+\epsilon(x))A_s^{(-1)}\widetilde{\mathbf{Q}}^{(0)}(x),\end{align}
where the first step follows by definition, the second step follows since the
matrix $A_{s}^{(-1)}$ is zero except for its bottom left entry,
cf.~\eqref{A0minus1:structure}, the third step is a consequence of
\eqref{def:eps}, and the last step uses that $\widetilde{\mathbf{Q}}^{(0)}(x)$
in \eqref{def:Qtil} has its first entry equal to $1$.

Inserting \eqref{def:Qtil2}--\eqref{def:Qtil3} in \eqref{recursion:mxv2} yields
the matrix-vector relation
\begin{equation}\label{nullspace}
x \widetilde{\mathbf{Q}}^{(0)}(x) = B(x,s)\widetilde{\mathbf{Q}}^{(0)}(x)
\end{equation}
where the matrix $B(x,s)$ satisfies
\begin{multline}\label{Bxs}
B(x,s)=\left(A_{s}^{(\beta)}z_1(x,s)^{\beta}(1+\mathcal O(x^{-l})) + \cdots
+A_{s}^{(1)}z_1(x,s)(1+\mathcal O(x^{-l})) + A_{s}^{(0)}\right.\\ \left. +
A_{s}^{(-1)}z_1(x,s)^{-1}(1+\epsilon(x))\right),\qquad x\to\infty.
\end{multline}
We can rewrite \eqref{Bxs} as
\begin{multline}\label{nullspacehip}
B(x,s) = \left(A_{s}^{(\beta)}z_1(x,s)^{\beta} + \cdots +A_{s}^{(1)}z_1(x,s) +
A_{s}^{(0)}+ A_{s}^{(-1)}z_1(x,s)^{-1}\right) \\ +
A_{s}^{(-1)}z_1(x,s)^{-1}\epsilon(x) +\mathcal O(x^{-l-r}),\qquad x\to\infty,
\end{multline}
by using that $z_1(x,s)=\mathcal{O}(x^{-r})$ as $x\to\infty$
(Lemma~\ref{lemma:mu0prob}(a)).

Relation \eqref{nullspace} clearly implies that
\begin{equation}\label{nullspace2}
\det \left(B(x,s)-x I_r\right) = 0.
\end{equation}
Expanding the determinant \eqref{nullspace2} for $|x|$ large, with the help of
\eqref{nullspacehip}, we obtain
\begin{multline}\label{det:exp}
\det\left(A_{s}^{(\beta)}z_1(x,s)^{\beta} + \cdots +A_{s}^{(1)}z_1(x,s) +
A_{s}^{(0)}+ A_{s}^{(-1)}z_1(x,s)^{-1}-xI_r\right)\\
+(-1)^{r+1}\epsilon(x)x^r(1+\mathcal O(1/x))+\mathcal O(x^{-l}) = 0,\qquad
x\to\infty.
\end{multline}
Here the terms in the second line of \eqref{det:exp} can be justified by using
the special structure of $A_{s}^{(0)}$ and $A_{s}^{(-1)}$ in
\eqref{A0minus1:structure}, and using again the fact that
$z_1(x,s)=x^{-r}(1+\mathcal O(1/x))$ as $x\to\infty$.

Now the determinant in the first line of \eqref{det:exp} vanishes identically,
since by definition $z_1=z_1(x,s)$ is a root of \eqref{fzlambda0},
cf.~\eqref{symbol0}. So \eqref{det:exp} reduces to $$
(-1)^{r+1}\epsilon(x)x^r(1+\mathcal O(1/x))+\mathcal O(x^{-l}) = 0,\qquad
x\to\infty,
$$ which implies in turn that $\epsilon(x) = \mathcal{O}(x^{-l-r})$ as
$x\to\infty$. This proves the induction step, thereby establishing the claim.

Having proved the claim, the proof of Lemma~\ref{lemma:Roman} can now be
finished from a standard normal family argument as in \cite{Roman}.
\end{proof}

With Lemma~\ref{lemma:Roman} in place, the proof of
Theorem~\ref{theorem:locallyToepeig} can be finished as in \cite[Proof of
Theorem 1.2]{Roman}. $\bol$

\begin{remark}
The above proof also shows that
$$\widetilde{\mathbf{Q}}^{(0)}(x) :=  \lim_{i\to\infty}
\frac{1}{Q_{rk_i,n_i}(x)}\begin{pmatrix}Q_{rk_i,n_i}(x)\\ \vdots \\
Q_{r(k_i+1)-1,n_i}(x)\end{pmatrix}$$ satisfies
\begin{equation*} x
\widetilde{\mathbf{Q}}^{(0)}(x) =
A_s(z_1(x))\widetilde{\mathbf{Q}}^{(0)}(x),\qquad x\in\cee\setminus [-R,R],
\end{equation*}
by virtue of \eqref{symbol0} and the fact that \eqref{nullspace}--\eqref{Bxs}
hold with $l$ arbitrarily large. So $\widetilde{\mathbf{Q}}^{(0)}(x)$ is a
vector with first component equal to $1$ which lies in the null space of the
matrix $A_s(z_1(x))-x I_r$. In fact, it can be shown that there is a unique
vector $\mathbf{v}(x)$ satisfying this condition, for all
$x\in\cee\setminus[-R,R]$, and hence we have
$$ \widetilde{\mathbf{Q}}^{(0)}(x) = \lim_{rk/n\to s}
\frac{1}{Q_{rk,n}(x)}\begin{pmatrix}Q_{rk,n}(x)\\ \vdots \\
Q_{r(k+1)-1,n}(x)\end{pmatrix} = \mathbf{v}(x),\qquad
x\in\cee\setminus[-R,R].$$
\end{remark}

\section{Proof of Theorem \ref{theorem:equilibriumproblem}}
\label{section:equilibrium}

\subsection{The sets $\Gamma_1(s)$ and $\Gamma_2(s)$}
Recall the functions $z_j =z_j(x,s)$,  defined as the solutions to the
algebraic equation (\ref{algebraicequation:P}) ordered by increasing modulus,
(\ref{roots}), and the
definitions 
(\ref{def:Gamma0}) and  (\ref{def:Gamma2}) of the sets
 $\Gamma_1(s)$ and
$\Gamma_2(s)$, whose structure we now describe.

\begin{proposition}\label{monotonoussupports}
$\Gamma_1(s)=[-c_1 s,c_1 s]$ and $\Gamma_2(s)=i\mathbb{R}\setminus
(-ic_2 s,ic_2 s)$, where
$c_1$ and $c_2$ are explicit constants given by (\ref{supportnu1}) and
(\ref{constraintactive}), respectively.
\end{proposition}

\begin{proof}
By (\ref{algebraicequation:P}), specializing (\ref{fzlambda0}) to the present setting gives  $$f_s(z,x)=x^2-P(z,s),$$ where
\begin{equation}\label{def:P}
    P(z,s)=\frac{(4+ z s^2 (1+b^2 ))^2(4+z s^2 b^2 )}{64z}, \qquad \text{where } b = \tan t.
\end{equation}
 It is clear that $x\in \Gamma_j(s)$ if and only if $-x\in
\Gamma_j(s)$, so it will be convenient to consider for a moment the
sets
\begin{equation}\Gamma_j^2(s):=\{ y \mid y=x^2, x\in \Gamma_j(s)\}.
\end{equation}

We begin by establishing that $\Gamma_1^2(s)\cup$ $\Gamma_2^2(s)\subset
\mathbb{R}$; the proof of this fact will be very similar to the proof of
Lemma~4.1 in \cite{Roman}. Suppose $y\in \Gamma_1^2(s)\cup \Gamma_2^2(s)$. We
can assume without loss of generality that $y$ is not a branch point, since the
number of branch points is finite and $\Gamma_1(s)$ and $\Gamma_2(s)$ in our
case have no isolated points \cite{Delvaux,Widom1}. Thus there exist distinct
$z_1,z_2\in \mathbb{C}$ such that $|z_1|=|z_2|=:r$ and $P(z_1,s)=y=P(z_2,s)$.
By the factorization of $P$, we see that $z \mapsto P(z,s)$ has only negative
real zeros and therefore the even function $[-\pi,\pi]\ni\theta \mapsto
|P(re^{i\theta},s)|$ is strictly decreasing on $(0,\pi)$, which implies in turn
that $z_1=\overline{z}_2$. But
\begin{equation*}
    y=P(z_2,s)=P(\overline{z}_1,s)=\overline{P(z_1,s)}=\overline{y},
\end{equation*}
so $y$ is real, and hence $\Gamma_1^2(s)\cup$ $\Gamma_2^2(s)\subset \mathbb{R}$.
This argument also shows that  $\Gamma_1^2(s) \cap \Gamma_2^2(s)$ may contain
only branch points, since otherwise there would be three distinct values $z_1, z_2, z_3$
with the same modulus and $y = P(z_1,s) = P(z_2,s) = P(z_3,s)$ which is clearly
impossible.


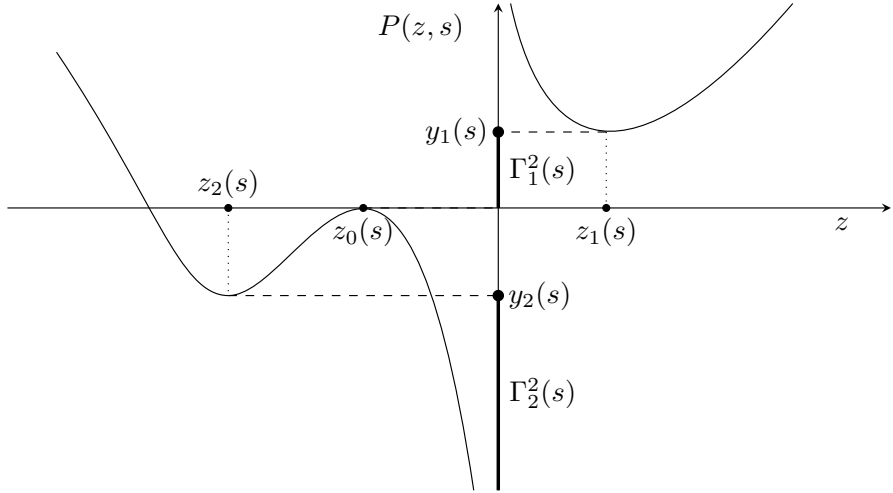
\begin{figure}
\begin{center}
\resizebox{.8\textwidth}{!}{
\begin{tikzpicture}[scale=0.6,>=stealth]
\draw[->] (-10,0.8) -- (8,0.8); \draw[->] (0,-5) -- (0,5);

\draw (-0.5,-5.0) .. controls (-2,5) and (-4,-1) ..  (-5.5,-1.0);

\draw (-5.5,-1) .. controls (-6.5,-1) and  (-7,1) .. (-9,4);

\draw (0.25,5) .. controls (1,1.5) and (3,1.5) ..  (6,5);

\draw[fill=black] (-5.5,0.8) circle (2pt) node[above] {$z_2(s)$};
\draw[style=dotted] (-5.5,0.8) -- (-5.5,-1);




\draw (-.5,4.5) node[left]{$P(z,s)$}; \draw (7,0.8) node[below]{$z$};

\draw[fill=black] (-2.75,0.8) circle (2pt) node[below] {$z_0(s)$};

\draw[fill=black] (2.2,0.8) circle (2pt) node[below] {$z_1(s)$};
\draw[style=dotted] (2.2,2.36) -- (2.2,0.8);

\draw[style=dashed] (2.2,2.36) -- (0,2.36);
\draw[style=dashed] (-2.75,0.8) -- (0,0.8);
\draw[style=dashed] (-5.5,-1) -- (0,-1);

\draw[fill=black] (0,2.36) circle (3pt) node[left]{$y_1(s)$};
\draw[fill=black] (0,-1) circle (3pt)node[right]{$y_2(s)$} ;

\draw[very thick] (0,0.8) -- (0,2.36) node[midway,right] {$\Gamma_1^2(s)$};
\draw[very thick] (0,-5) -- (0,-1) node[midway,right] {$\Gamma_2^2(s)$};
\end{tikzpicture}}
\end{center}
\caption{Plot of $P(z,s)$, $z\in \mathbb{R}$.}\label{Plot_symbol}
\end{figure}

Studying the function $\mathbb{R}\ni z\mapsto P(z,s)$ for fixed $s>0$, we see
that it has two local minima at the points
\begin{equation*}
z_{1} =\frac{1}{b^2}\left(-1+\sqrt{\frac{9b^2+1}{b^2+1}}\right)s^{-2},
\end{equation*}
and
 \begin{equation} \label{z2s}
 z_{2}=-\frac{1}{b^2}\left(1+\sqrt{\frac{9b^2+1}{b^2+1}}\right)s^{-2},
 \end{equation}
 with
\begin{equation}\label{squaredsupport1}
y_{1}(s):=P(z_1,s)=s^2\left(\frac{27b^4+18b^2-1+\sqrt{b^2+1}(9b^2+1)^{3/2}}{32b^2}\right)=(c_1s)^2\geq
0,
\end{equation}
and
\begin{equation}\label{squaredsupport2}
y_{2}(s):=P(z_2,s)=s^2\left(\frac{27b^4+18b^2-1-\sqrt{b^2+1}(9b^2+1)^{3/2}}{32b^2}\right)=(ic_2 s)^2\leq
0,
\end{equation}
and that $0$ is a local maximum attained at the point
\begin{equation*}z_0(s)=-\frac{4}{1+b^2}s^{-2}
\end{equation*}
(see Fig. \ref{Plot_symbol}).

Since
\begin{equation}\label{algebraicequation:P2}
P(z,s)=y
\end{equation}
is a polynomial equation in $z$ with real coefficients, it has two complex
conjugate solutions if $y\in (-\infty,y_2(s))\cup (0,y_1(s))$. If
$y\in\er\setminus\left( (-\infty,y_2(s)]\cup [0,y_1(s)]\right)$, then all three
solutions to \eqref{algebraicequation:P2} are real and it is easy to see that
there can be at most a finite number of such $y$ for which two of these roots
have the same modulus. But as already mentioned, $\Gamma_1(s)$ and
$\Gamma_2(s)$ cannot have isolated points, hence
$\Gamma_1^2(s)\cup\Gamma_2^2(s)\subset (-\infty,y_2(s)]\cup [0,y_1(s)]$.

Now consider the interval $[0,y_1(s)]$. For the branch point $y=y_1(s)$, we have
that $z_1(s)>0$
is a double root to \eqref{algebraicequation:P2}, and there is also a negative
root $z_-(s)$ which is smaller than $z_2(s)$ (see Fig.~\ref{Plot_symbol}).
Therefore
\begin{equation*}
|z_1(s)|-|z_-(s)|< z_1(s)+z_2(s) =-\frac{2}{b^2s^2}<0,
\end{equation*}
i.e., the real negative root $z_-(s)$ has larger modulus than the double
positive root $z_1(s)$, so $y_1(s)\in \Gamma_1^2(s) \setminus \Gamma_2^2(s)$.
In a similar way it follows that $0 \in \Gamma_1^2(s) \setminus \Gamma_2^2(s)$
and $y_2(s) \in \Gamma_2^2(s) \setminus \Gamma_1^2(s)$. Thus the branch points are
not in $\Gamma_1^2(s) \cap \Gamma_2^2(s)$, and therefore
\[ \Gamma_1^2(s) \cap \Gamma_2^2(s) = \emptyset, \]
since we already proved that non-branch points are not in the intersection.

We have now established that
\[ \Gamma_1^2(s) \cup \Gamma_2^2(s) = [0, (c_1s)^2] \cup (-\infty, -(c_2s)^2] \]
with $0, (c_1 s)^2 \in \Gamma_1^2(s)$ and $-(c_2s)^2 \in \Gamma_2^2(s)$.
We also established that $\Gamma_1^2(s) \cap \Gamma_2^2(s) = \emptyset$, and then
it follows by continuity that
\[ \Gamma_1^2(s) = [0, (c_1s)^2], \qquad \Gamma_2^2(s) = (-\infty, -(c_2s)^2]. \]
This completes the proof.


\end{proof}

\subsection{Equilibrium problem}
\label{subsection:equilibriumproblem2}
Given a measure $\mu$ in
the complex plane, define the \emph{logarithmic potential} of $\mu$,
\begin{equation*}
\mathcal{U}^{\mu}(x)=\int \log\frac{1}{|y-x|}\,\ud \mu (y).
\end{equation*}
For  $j=1,2$, define the measures
\begin{equation}\label{explicitdensity}
\ud \mu_j^s(x)= \frac{1}{2}\frac{1}{2\pi i}
\left(\frac{z_{j+}' (x,s)}{z_{j+}(x,s)}
-\frac{z_{j-}'(x,s)}{z_{j-}^{}(x,s)}\right)\,\ud x,
\end{equation}
where $\ud x$ is the (complex) line element on 
$\Gamma_j(s)$. 

\begin{proposition}\label{equilibriumdelvaux} 
Fix $s\geq 0$. The pair $(\mu_1^s,\mu_2^s)$ is the unique minimizer
 of the energy functional
\begin{equation*}
J(\mu,\nu)=I(\mu,\mu)+I(\nu,\nu)-I(\mu,\nu)
\end{equation*}
among all pairs $(\mu,\nu)$ of positive measures such that $\operatorname{supp}
\mu \subset \Gamma_1(s)$, $\int \ud \mu=1$; $\operatorname{supp} \nu \subset
\Gamma_2(s)$ and $\int \ud \nu=1/2$.

For all $x\in \mathbb{C}$, $(\mu_1^s,\mu_2^s)$ satisfies the Euler-Lagrange variational conditions

 \begin{equation}\label{explicitequilibrium}
\left\{\begin{array}{ll}
2\mathcal{U}^{\mu_1^s}(x)-\mathcal{U}^{\mu_2^s}(x)-l^{s}=-\frac{1}{2}\log
\displaystyle \frac{|z_2(x,s)|}{|z_1(x,s)|},  \\
-\mathcal{U}^{\mu_1^s}(x)+2\mathcal{U}^{\mu_2^s}(x)=-\frac{1}{2}\log
\displaystyle \frac{|z_3(x,s)|}{|z_2(x,s)|},
\end{array}
\right.
\end{equation}
for some constant $l^s$.
\end{proposition}
\begin{proof} See \cite{Delvaux}.
\end{proof}

We will now integrate
(\ref{explicitequilibrium}) to get an equilibrium problem for $(\nu_1,\nu_2)$,
where
\begin{equation}\nu_2:=\int_0^1\mu_2^s\,\ud s,
\end{equation}
in analogy with the definition of $\nu_1$ in (\ref{def:nu_1}).
\begin{proposition}\label{integratedproposition}
For all complex $x$, the vector $(\nu_1,\nu_2)$ of measures satisfies the
following conditions. Firstly,
\begin{equation}\label{integratedequilibrium1}
2\mathcal{U}^{\nu_1}(x)-\mathcal{U}^{\nu_2}(x)-l+V(x)
 \geq 0, \ \ 
\end{equation}
where
\begin{equation*}
V(x)=\frac{1}{2}\int_{0}^{\infty} \log \frac{|z_2(x,s)|}{|z_1(x,s)|}\,\ud s
\end{equation*}
and $l$ is some constant, with equality in \eqref{integratedequilibrium1} if
and only if $x\in[-c_1,c_1]$. Secondly,
\begin{equation}\label{integratedequilibrium2}
-\mathcal{U}^{\nu_1}(x)+2\mathcal{U}^{\nu_2}(x)\geq 0,
\end{equation}
with equality if and only if $x\in i\mathbb{R}\setminus (-ic_2,ic_2)$.

Furthermore, $\nu_2\leq \sigma$ where
\begin{equation}\label{def:sigma}
 \sigma=\int_0^{\infty}\,\mu_2^s\,\ud s.
\end{equation}
\end{proposition}
\begin{proof}
By Proposition \ref{monotonoussupports}, the supports of the measures $\mu_1^s$
and $\mu_2^s$ are subsets of the real and imaginary lines respectively, which
increase/decrease linearly in $s$. This means that we can integrate the
logarithmic potentials of $\mu_1^s$ and $\mu_2^s$  with respect to $s$ and
change the order of integration to obtain
\begin{align*}
  \int_0^1 \mathcal{U}^{\mu_1^s}(x) \ud s
    & =\int_0^1 \int_{-c_1s}^{c_1s} \log \frac{1}{|y-x|}\ud \mu_1^s(y) \ud s \\
   & =\int_{-c_1}^{c_1}\int_0^{|y|/c_1} \log\frac{1}{|y-x|}\,\ud s\,\ud \mu_1^s (y)
  =:\mathcal{U}^{\nu_1}(x)
 \end{align*}
and similarly for $\nu_2$.
The integrated variational conditions  (\ref{explicitequilibrium}) thus become
\begin{equation}\label{integratedequilibrium1a}
2\mathcal{U}^{\nu_1}(x)-\mathcal{U}^{\nu_2}(x)-\int_0^s l^s\ud s=-\frac{1}{2}\int_0^1\log \frac{|z_2(x,s)|}{|z_1(x,s)|}\ud s, \ \ 
\end{equation}
\begin{equation}\label{integratedequilibrium2a}
-\mathcal{U}^{\nu_1}(x)+2\mathcal{U}^{\nu_2}(x)=-\frac{1}{2}\int_0^1\log \frac{|z_3(x,s)|}{|z_2(x,s)|}\ud s. \ \ 
\end{equation}
Furthermore, since by definition $|z_2(x,s)|\geq |z_1(x,s)|$,
\begin{equation*}
\frac{1}{2}\int_0^1\log \frac{|z_2(x,s)|}{|z_1(x,s)|}\ud s \leq \frac{1}{2}\int_0^{\infty}\log \frac{|z_2(x,s)|}{|z_1(x,s)|}\ud s =V(x)
\end{equation*}
with equality if and only if $x\in \Gamma_1(1) = [-c_1,c_1]$. Clearly, $\nu_2$ must satisfy
\begin{equation}\label{confusing:ref}
\frac{\ud \nu_2}{|\ud x|}=\int_0^1\,\frac{\ud \mu_2^s}{|\ud x|}\,\ud s\leq
\int_0^{|x|/c_2}\,\frac{\ud \mu_2^s}{|\ud x|}\,\ud s =\frac{\ud \sigma
(x)}{|\ud x|},
\end{equation}
since $x \not\in
\operatorname{supp} \mu_2^s$ for $s>|x|/c_2$, with 
equality in \eqref{confusing:ref} if and only if 
$x \in \Gamma_2(1)$. Inserting into  (\ref{integratedequilibrium1a}) and
(\ref{integratedequilibrium2a}) and putting $l:=\int_0^s l^s\ud s$  gives the
stated inequalities.

\end{proof}
Equations (\ref{integratedequilibrium1}) and (\ref{integratedequilibrium2}) are
the Euler-Lagrange variational conditions for the equilibrium problem in
Theorem \ref{theorem:equilibriumproblem}. The external field $V$, density of
the measure $\nu_1$ and upper constraint measure $\sigma$ can be calculated
explicitly, and the following subsections are devoted to these computations.

\subsection{Calculation of $V$}

In the calculations that follow we make use of the function
\begin{equation} \label{defQ}
    Q(z) := \frac{(4+(1+b^2) z)^2(4+b^2 z)}{64z},
    \end{equation}
which is such that
\begin{equation} \label{QandP}
    Q(z) = \frac{1}{s^2} P(z/s^2;s)
    \end{equation}
for every $s > 0$. So in particular $Q(z) = P(z;1)$.
We also define
\[ \tilde{z}_j(x) := z_j(x,1), \qquad j=1,2,3, \]
and these are the solutions of $Q(z) = x^2$. Because of \eqref{QandP}
we have
\begin{equation} \label{zjandtildezj}
    z_j(x,s) = \frac{1}{s^2} \tilde{z}_j\left(\frac{x}{s} \right), \qquad s > 0,
    \end{equation}
and it follows that
\begin{equation} \label{logderzj}
    \frac{1}{z_j(x,s)} \frac{\partial z_j(x,s)}{\partial x} =
    \frac{\tilde{z}_j'(x/s)}{s \tilde{z}_j(x/s)}, \qquad s > 0, \quad j=1,2,3.
    \end{equation}

\begin{proposition} \label{extfieldV}
The external field $V$ is given by
\begin{equation}
    V(x) = \frac{1}{2} \int_0^{\infty} \log \frac{|z_2(x,s)|}{|z_1(x,s)|} ds = (\pi - 2t) |x|,
        \qquad x \in \mathbb R.
        \end{equation}
\end{proposition}
\begin{proof}
Because of symmetry, we may assume $x > 0$.

It follows from Proposition \ref{monotonoussupports} that $\Gamma_1(s)$ is increasing
with $s$ and $x \in \Gamma_1(s)$ if and only if $s \geq x/c_1$. Thus $|z_2(x,s)| = |z_1(x,s)|$
for $s \geq x/c_1$, and the integral that defines $V(x)$ can be restricted to an integral
over $s \in [0, x/c_1]$. Using \eqref{logderzj} we obtain
\begin{align}
    \frac{\ud V(x)}{\ud x} &=  \frac{1}{2} \int_0^{x/c_1}
    \left(\frac{1}{z_2(x,s)}\frac{\partial z_2(x,s)}{\partial x} -
          \frac{1}{z_1(x,s)}\frac{\partial z_1(x,s)}{\partial x} \right) \ud s \nonumber\\
    &= \frac{1}{2} \int_0^{x/c_1} \left( \frac{\tilde{z}_2'(x/s)}{s \tilde{z}_2(x/s)} -
    \frac{\tilde{z}_1'(x/s)}{s \tilde{z}_1(x/s)} \right) \ud s \nonumber \\
    & = \frac{1}{2} \int_{c_1}^{\infty} \left(\frac{\tilde{z}_2'(u)}{u \tilde{z}_2(u)} -
    \frac{\tilde{z}_1'(u)}{u \tilde{z}_1(u)} \right) \ud u
    \label{dvdx}
    \end{align}
where in the last step we made the change of variables $u = x/s$ (recall that $x > 0$).
Note that \eqref{dvdx} does not depend on $x$. To evaluate \eqref{dvdx} we note that
both $u \mapsto \tilde{z}_1(u)$ and $u \mapsto \tilde{z}_2(u)$ are one-to-one for $u \in [c_1,\infty)$
and they map the interval $[c_1, \infty)$ onto $(0, \tilde{z}_1(c_1)]$ and $[\tilde{z}_2(c_1), \infty)$,
respectively. We split the integral  \eqref{dvdx} into two integrals, and apply
a change of variables $z = \tilde{z}_j(u)$ to each of them. Then combining the
two integrals again, and noting that $\tilde{z}_1(c_1) = \tilde{z}_2(c_1)$ and that
$u = \sqrt{Q(z)}$ if  $z = \tilde{z}_j(u)$ with $j=1,2$, we obtain
\begin{align} \label{dvdx2}
    \frac{\ud V(x)}{\ud x}  = \frac{1}{2} \int_0^{\infty}
    \frac{\ud z}{z \sqrt{Q(z)}}, \qquad x > 0.
    \end{align}

The integral in \eqref{dvdx2} can be calculated explicitly, since
\begin{equation} \label{primitive}
    \frac{\ud}{\ud z} \left[ \arctan \left( \frac{4 + (b^2-1)z}{2\sqrt{z(4+b^2 z)}}\right) \right]
    = \frac{4}{(4 + (1+b^2)z) \sqrt{z(4+b^2z)}} = - \frac{1}{2 z \sqrt{Q(z)}}.
    \end{equation}
Therefore
\begin{align*}
    \frac{\ud V(x)}{\ud x} & = \lim_{z\to 0+} \arctan \left( \frac{4 + (b^2-1)z}{2\sqrt{z(4+b^2 z)}}\right)
        - \lim_{z \to +\infty} \arctan \left( \frac{4 + (b^2-1)z}{2\sqrt{z(4+b^2 z)}}\right) \\
        & = \frac{\pi}{2} - \arctan \left( \frac{b^2-1}{2b} \right).
        \end{align*}
Using that $ b= \tan t$ and applying trigonometric identities, we finally obtain
\begin{equation} \label{dvdx3}
    \frac{\ud V}{\ud x} = \pi - 2t.
    \end{equation}
Noting that $V(0) = 0$ since $0 \in \Gamma_1(s)$ for every $s > 0$, we find
the claimed expression for the external field by integrating \eqref{dvdx3}
with respect to $x$.
\end{proof}

\begin{remark}
Note that the external field $V$ has the form to be expected by analyzing
directly the asymptotics of the weight functions $w_i$: For the rescaled
polynomials $Q_{k,n}$ the orthogonality conditions read
\begin{equation*}
\int_{-\infty}^{\infty} Q_{k,n}(x)x^m w_j(nx)\,\ud x= 0, \quad m=0,\ldots,k_j-1,\quad
j=1,2.
\end{equation*}
Thus we have new effective weights $\tilde{w}_j(x)=w_j(nx)$. Using Stirling's formula,
\begin{align*}
\tilde{w}_j(x)&=\frac{1}{2\pi}e^{2t_jnx}\left|\Gamma\left(\lambda+inx\right)\right|^2\\
&=e^{-2\lambda}\left|\lambda+inx\right|^{2\lambda-1}e^{2t_jnx
-2nx\arg \left(\lambda+inx\right)}(1+o(1))\\
&=e^{n(2t_jx - \pi |x|)(1+o(1))}.
\end{align*}
Asymptotically, the dominant weight determining the potential associated with
the distribution of zeros, will be
$\max\{\tilde{w}_1(x),\tilde{w}_2(x)\}=e^{-n\left(\pi-2t\right)|x|(1 +o(1))}$,
giving the external field $V(x)=\left(\pi-2t\right)|x|$.
\end{remark}

\subsection{Density of $\nu_1$}

Next we turn to the density of $\nu_1$.
\begin{proposition}\label{propdensity}
The measure $\nu_1$ is absolutely continuous with density given by
\eqref{densityrel1bis}--\eqref{densityrel3}.
\end{proposition}

\begin{proof}
By \eqref{def:nu_1} and (\ref{explicitdensity}),
\begin{equation*}
\frac{\ud \nu_1}{\ud x} = \frac{1}{2}\frac{1}{2\pi i}\int_{|x|/c_1}^1\,
    \left(\frac{z_{1+}'(x,s)}{z_{1+}(x,s)}
 -\frac{z_{1-}'(x,s)}{z_{1-}(x,s)}\right)\,\ud s, \qquad x \in [-c_1,c_1].
\end{equation*}
We now make essentially the same calculations as in the proof of Proposition
\ref{extfieldV}. Assuming $x \in (0,c_1)$ and using \eqref{logderzj} we obtain
as in \eqref{dvdx}
\begin{align*}
\frac{\ud \nu_1}{\ud x}(x)
        = \frac{1}{2} \frac{1}{2\pi i} \int_x^{c_1} \left(
            \frac{\tilde{z}_{1+}'(u)}{u \tilde{z}_{1+}(u)} -
            \frac{\tilde{z}_{1-}'(u)}{u \tilde{z}_{1-}(u)} \right) \ud u.
            \end{align*}
A change of variables $z = \tilde{z}_{1\pm}(u)$ leads to
\begin{equation} \label{dnu1dx2} \frac{\ud \nu_1}{\ud x}(x) =
    \frac{1}{2} \frac{1}{2\pi i} \int_{\tilde{z}_{1+}(x)}^{\tilde{z}_{1-}(x)} \frac{\ud z}{z \sqrt{Q(z)}},
    \end{equation}
which is an integral in the complex $z$-plane.
Since $x \in (0, c_1)$, we have that $\tilde{z}_{1+}(x)$ and $\tilde{z}_{1-}(x)$ are each others
complex conjugate, and it can be shown that $\operatorname{Im} \tilde{z}_{1+}(x) < 0$. The integral in \eqref{dnu1dx2}
is from $\tilde{z}_{1+}(x)$ in the lower half plane to its complex conjugate in the upper half plane
along a path in $\mathbb C \setminus (-\infty, 0]$. The square root $\sqrt{Q(z)}$ is defined and
analytic in $\mathbb C \setminus (-4 b^{-2}, 0]$ and it is positive for $z$ real and positive.

Let
\begin{equation} \label{wx2}
    w(x)  =  \frac{4 + (b^2-1)z}{2 \sqrt{z(4 + b^2z)}}, \qquad z = \tilde{z}_{1-}(x), \qquad 0 < x < c_1.
    \end{equation}
Then by  \eqref{primitive} and \eqref{dnu1dx2}
\[ \frac{\ud \nu_1}{\ud x}(x) =
    \frac{1}{2\pi i} \left( \arctan ( \overline{w(x)}) - \arctan (w(x)) \right). \]
The arctangent is understood here as an analytic function
\[ \arctan w = \frac{i}{2}  \log \left( \frac{1-iw}{1+iw} \right),
    \qquad w \in \mathbb C \setminus ((-i \infty, -i) \cup (i, i\infty)). \]
Thus
\begin{align*}
    \frac{\ud \nu_1}{\ud x}(x) & = \frac{1}{4\pi i}
        \log \left( \frac{1 - i \overline{w(x)}}{1+i \overline{w(x)}} \cdot \frac{1+iw(x)}{1-iw(x)}    \right) \\
        & = \frac{1}{2\pi} \log \left|\frac{ 1 + i w(x)}{1-iw(x)} \right|.
        \end{align*}
Using the defining relation $Q(\tilde{z}_{1-}(x)) = x^2$ in \eqref{wx2} we can see that
\eqref{wx2} is equal to \eqref{densityrel2}--\eqref{densityrel3} and then \eqref{densityrel1bis} follows.
Note also that $\Im w(x) < 0$ since $\nu_1$ has a positive density.
\end{proof}

\subsection{Calculation of upper constraint measure}

\begin{proposition}\label{propconstraint}
The upper constraint measure $\sigma$ in \eqref{def:sigma} is a multiple of the
Lebesgue measure on the imaginary axis with density $2t/\pi$.
\end{proposition}
\begin{proof}
Because of symmetry it is enough to consider the density on the positive
imaginary axis. Let $x \in i \mathbb R^+$.
Then
\[ \frac{\ud \sigma}{|\ud x|}(x) = i \frac{\ud \sigma}{\ud x}(x)
    = \frac{1}{4 \pi} \int_0^{|x|/c_2}
        \left( \frac{z'_{2+}(x,s)}{z_{2+}(x,s)} - \frac{z'_{2-}(x,s)}{z_{2-}(x,s)} \right) \ud s \]
and by a calculation as in the proof of Proposition \ref{extfieldV} this is equal to
\[ \frac{\ud \sigma}{|\ud x|}(x)
    = \frac{1}{4\pi} \int_{ic_2}^{i\infty}
        \left( \frac{\tilde{z}'_{2+}(u)}{u \tilde{z}_{2+}(u)} - \frac{\tilde{z}'_{2-}(u)}{u\tilde{z}_{2-}(u)} \right) \ud u. \]

Now we are going to make the change of variables $z = \tilde{z}_{2\pm}(u)$. Since
$z_{2+}(i\infty) = - i \infty$, $z_{2-}(i \infty) = +i \infty$, and $\tilde{z}_{2+}(ic_2) = \tilde{z}_{2-}(ic_2)$,
we obtain
\begin{equation} \label{dsigmadx2}
    \frac{\ud \sigma}{|\ud x|}(x)
    = -\frac{1}{4\pi} \int_{-i \infty}^{+i\infty} \frac{\ud z}{z (Q(z))^{1/2}}
    \end{equation}
with integration along a contour from $-i \infty$ to $+i \infty$ that intersects the real
line in
\[ \tilde{z}_{2 \pm}(ic_2) =  z_2(1) = - \frac{1}{b^2} \left( 1 + \sqrt{\frac{9b^2+1}{b^2+1}} \right), \]
see \eqref{z2s} and Figure \ref{Plot_symbol} for $z_2(s)$.
It is easy to check (and this can also be seen from Figure \ref{Plot_symbol}) that
\[ -\frac{4}{b^2} <  - \frac{1}{b^2} \left( 1 + \sqrt{\frac{9b^2+1}{b^2+1}} \right) < - \frac{4}{b^2+1}. \]
and so the contour intersects the real line in a point lying in between the simple root and the double root
of $Q(z)$.
The branch of $(Q(z))^{1/2}$ that is used in \eqref{dsigmadx2} is the one that is
defined and analytic in $\mathbb C \setminus (-\infty, -4/b^2] \cup [0, \infty)$ and that is
in $i \mathbb R^+$ for $z = \tilde{z}_{2\pm}(ic_2)$.
Thus
\begin{equation} \label{Qzsqrt}
    (Q(z))^{1/2} = -i (4 + (1+b^2) z)  \frac{(4 + b^2 z)^{1/2}}{8(-z)^{1/2}}
    \end{equation}
with principal branches of the fractional powers.

We now deform the contour in \eqref{dsigmadx2} to the positive real axis.
In doing so, we pick up a residue contribution at $z = - \frac{4}{b^2+1}$, which is
\begin{equation} \label{polecontribution}
    \frac{1}{4\pi} (2 \pi i) \operatorname{Res}\left(\frac{1}{z (Q(z))^{1/2}}, z = - \frac{4}{b^2+1} \right) =  1,
    \end{equation}
since the residue turns out to be $-2i$, and so is independent of $b$.
The deformed contour goes along the positive real axis, starting at $+\infty$ on the lower side,
and ending at $+\infty$ on the upper side of the real axis. It gives the contribution
\[ - \frac{1}{4\pi} \int_0^{+\infty} \left(\frac{1}{z (Q(z))^{1/2}_+} - \frac{1}{z (Q(z))^{1/2}_-} \right) \ud z. \]
With our choice \eqref{Qzsqrt} of square root we have $(Q(z))^{1/2}_+ =
-(Q(z))^{1/2}_- = \sqrt{Q(z)}$ where $\sqrt{\cdot}$ is the positive square root
of a real and positive number. So the contribution from  the positive real line
is the integral
\[ - \frac{1}{2\pi} \int_0^{+\infty} \frac{\ud z}{z \sqrt{Q(z)}} \]
and this integral we already calculated in \eqref{dvdx2} and \eqref{dvdx3}. Its
value is
\begin{equation} \label{linecontribution}
    - \frac{1}{2\pi} \int_0^{+\infty} \frac{\ud z}{z \sqrt{Q(z)}} = -1 + \frac{2t}{\pi}.
    \end{equation}
Adding \eqref{polecontribution} and \eqref{linecontribution} we find
\[ \frac{d\sigma}{|dx|}= \frac{2t}{\pi} \]
and this proves the proposition.
\end{proof}

\subsection{Proof of Theorem \ref{theorem:equilibriumproblem}}
The equilibrium problem has a unique minimizer which will satisfy the
Euler-Lagrange variational conditions, and Proposition
\ref{integratedproposition} shows that $(\nu_1,\nu_2)$ has this property. The
explicit expressions for the external field, upper constraint measure and
density of $\nu_1$ are calculated in Propositions \ref{extfieldV},
\ref{propconstraint} and \ref{propdensity}. $\bol$

\section*{Acknowledgements}
We thank Pavel Bleher, Filippo Colomo and Karl Liechty for useful
discussions and pointers to the literature.

\end{document}